\documentclass[reqno, 11pt,oneside]{amsart}
\usepackage{amsmath,amsthm}
\usepackage{color,enumerate}
\usepackage{amssymb}
\usepackage{upgreek}
\usepackage{mathrsfs}
\usepackage{mathabx}
\usepackage{txfonts}
\usepackage{graphicx}
\usepackage{enumitem}		% special itemize with custom tags and labels that work 
\usepackage{soul}
\usepackage[colorlinks=true,linkcolor=blue, urlcolor=red, citecolor=blue]{hyperref}	

\usepackage[normalem]{ulem}
\definecolor{forestgreen(web)}{rgb}{0.13, 0.55, 0.13}

\newcommand{\vep}{\varepsilon}

\theoremstyle{plain}
\newtheorem{maintheorem}{Theorem}

\newtheorem{maincorollary}[maintheorem]{Corollary}

\newtheorem{theorem}{Theorem}[section]
\newtheorem{proposition}{Proposition}[section]
\newtheorem{lemma}{Lemma}[section]
\newtheorem{corollary}{Corollary}[section]
\newtheorem{definition}{Definition}[section]
\newtheorem{remark}{Remark}[section]
\newtheorem{example}{Example}[section]

\def\R{\mathbb{R}}

\definecolor{bgreen}{rgb}{0.13, 0.55, 0.13}

\begin{document}

\title
%[Typical impulsive dynamical systems]
%{The general density theorem for impulsive semiflows}
%{The non-wandering set of typical impulsive dynamical systems is \color{red} compact \color{black} and invariant}
%[Topological aspects of $C^1$-generic impulsive flows]{Topological aspects of $C^1$-generic impulsive flows}
[Abundance of periodic orbits for typical impulsive flows]{Abundance of periodic orbits for typical \\ impulsive flows}

 \author[J. Siqueira]{Jaqueline Siqueira}
\address{Jaqueline Siqueira\\ Departamento de Matem\'atica, Instituto de Matem\'atica, Universidade Federal do Rio de Janeiro, Caixa Postal 68530, Rio de Janeiro, RJ 21945-970, Brazil}
\email{jaqueline@im.ufrj.br}

\author[M. J. Torres]{Maria Joana Torres}
\address{CMAT and Departamento de Matem\'atica, Universidade do Minho, 
Campus de Gualtar,
4700-057 Braga, Portugal}
\email{jtorres@math.uminho.pt}

\author[P. Varandas]{Paulo Varandas}
\address{Departamento de Matem\'atica, Universidade Federal da Bahia\\
Av. Ademar de Barros s/n, 40170-110 Salvador, Brazil $\&$ CMUP, University of Porto -Portugal}
\email{paulo.varandas@ufba.br}

\date{\today}

\keywords{Impulsive semiflows; Periodic points; Non-wandering set}
\subjclass[2010]{Primary: 37A05, 37A35.}

\begin{abstract}
Impulsive dynamical systems, modeled by a continuous 
semiflow and an impulse function, may be discontinuous
and may have non-intuitive topological properties, as the non-invariance of the non-wandering set or the non-existence of invariant probability measures.
In this paper we study dynamical features of impulsive flows parameterized by the space of impulses. 
We prove that impulsive semiflows determined by a $C^1$-Baire generic impulse are such that the set of hyperbolic periodic orbits is dense in the 
set of non-wandering points which meet the impulsive region.  As a consequence, we provide sufficient conditions for the non-wandering set of a typical impulsive semiflow (except the discontinuity set) to be invariant. 
Several applications are given concerning impulsive semiflows obtained from billiard, Anosov and geometric Lorenz flows.  
\end{abstract}

\maketitle

%\tableofcontents

\section{Introduction}

Dynamical systems, as a field of study, aim to describe the evolution of a system over time. The classical theory of dynamical systems has proven invaluable in describing a myriad of phenomena, from celestial mechanics to population dynamics. However, many real-world scenarios involve sudden and discontinuous changes, a paradigm often overlooked by conventional dynamical systems theories. 
Some examples include abrupt perturbations in population dynamics caused by pandemics,  abrupt perturbations of billiard flows 
caused by inelastic collisions, or impulsive perturbations of geometric Lorenz attractors that 
model atmospheric convection (and describe partially changes in the Earth's climate system) due to drastic changes caused by 
the eruption of a volcano, just to mention a few examples.
%is highly complex, involving interactions between the atmosphere, oceans, land surface, and other components. General Circulation Models (GCMs) and other climate models are typically more sophisticated and comprehensive than the Lorenz equations. These models incorporate a wide range of physical processes, including fluid dynamics, thermodynamics, radiative transfer, and more, to simulate the behavior of the climate system on a global scale.
In general such dynamical systems may have discontinuities, which are a source of complexity.  
%discontinuous, and for that reason they do not fit in the classical theory of topological dynamical systems.  This creates many challenges, as 
In particular,
there are examples of impulsive dynamical systems acting on compact metric spaces 
which have no invariant measures and for which the non-wandering set is not forward invariant. One may inquire if this is the typical case,
and the main goal of this paper is to contribute to a better understanding of these problems.

\medskip
Historically, impulses have been first considered and much studied by the differential equation communities, leading to the notion of 
impulsive differential equations.  This paper fits in the context of impulsive dynamical systems, which gained interest in more recent years. Impulsive dynamical systems are modeled by three objects: a continuous semiflow acting on a  compact metric space $(X, d)$; a %\color{red} \color{red} compact \color{black} \color{black} 
set $D\subset X$ where the semiflow suffers abrupt perturbations; and an impulsive function $I:D\to X $ which models such abrupt perturbations.
%The impulsive map may lead to the creation of discontinuities,  and this is the reason why one may lose many of the very basic properties of classical continuous dynamical systems. 
The standard assumption that an impulse $I$ satisfies $I(D)\cap D=\emptyset$ is often responsible for the creation of discontinuities
(cf. Section~\ref{sec:examples} for examples).

\medskip
%For instance, impulsive dynamical systems need not have invariant probability measures
From the topological viewpoint,  there is a number of contributions to the study of global attractors {topological features} and global attractors for impulsive systems (see e.g.  \cite{BBCC, BP, CC04} and references therein).  
%For instance,  the concept of global attractors need to be redefined in the context of impulsive dynamics, since the sets with the property expected for an attractor fail to be \color{red} compact \color{black} in general \cite{BBCC, BP}.   
The ergodic aspects of impulsive dynamical systems seem to be much less studied,
and most contributions either provide sufficient conditions for the existence of invariant probabilities or study sufficient conditions, similar to expansiveness and specification, for
the impulsive semiflow to admit equilibrium states (cf. 
%\cite{AC14,ACS17,BR20,JSM19,JSM}
\cite{AC14,ACS17,BR20,JSM} and references therein). 
%However, there are known examples of impulsive dynamical systems which have no invariant probabilities. 
% 

 \medskip
In this paper we focus on some topological aspects of impulsive semiflows, namely the description of their non-wandering set and hyperbolicity features. 
While there exist impulsive dynamical systems whose non-wandering set is not invariant under the action of the impulsive semiflow and have no periodic orbits (see e.g. Example~\ref{ex:ACrobusto}),  a natural question %Our aim in this note 
is to understand how typical this behavior can be. For instance, 
inspired by Pugh's general density theorem for $C^1$-diffeomorphisms ~\cite{Pu}, one can ask whether $C^1$-typical impulses (meaning a $C^1$-Baire generic subset of such impulsive maps) lead to impulsive semiflows with good topological and ergodic features, including the density of periodic orbits in the non-wandering set. Notice that an affirmative answer will indicate that instead of preventing periodic behavior, typical abrupt changes could promote abundant periodic behavior and sensitivity to initial conditions.

 \medskip
The main result of this paper answers affirmatively to the previous question in case the impulses are taken as $C^1$-diffeomorphisms onto their image.  Indeed, we prove that for a $C^1$-smooth  flow on a compact manifold $M$, a $C^1$-typical impulse $I:D\to M$ generates an impulsive 
semiflow $\psi_I$  so that % the set $\Omega({\psi_I})\setminus D$, consisting of the non-wandering points that are not in the discontinuity region,
%is $\psi_I$-invariant, and so that
 the closure of the set of hyperbolic periodic points for $\psi_I$ 
 is dense in the set of non-wandering orbits which intersect the impulsive set (see Theorem~\ref{thmA} for the precise statement).
This is the optimal statement %One cannot expect to obtain a full strength 
of a general density theorem 
for impulsive semiflows because the impulsive region is not assumed to be a
global cross-section to the original flow, hence it may exist part of the non-wandering set for the original flow which is not affected by perturbations of the impulse (see e.g. Example~\ref{ex:prey}). \ Moreover, no hyperbolicity is assumed whatsoever. \color{black}
In Corollary~\ref{corA} we give %checkable 
sufficient conditions, either on the original flow or the initial impulse of an impulsive semiflow, which guarantee the everywhere invariance of the non-wandering set except on the discontinuity region for typical $C^1$-perturbations of such impulse. 
Bearing in mind Pugh's strategy for the proof of the general density theorem for diffeomorphisms, we believe that one of the main contributions of this paper is also of a technical nature, namely to develop tools (involving uniform hyperbolicity and perturbative results for discontinuous semiflows) which are of independent interest and will certainly be used in future study of impulsive dynamical systems.

% \item[$\circ$] the semiflow $\psi$ satisfies the shadowing property,
%% \item[$\circ$] there exist $\psi$-invariant probabilities $\mu$,
% \item[$\circ$] the space of ergodic probabilities is dense in the space of invariant probabilities, and
% \item[$\circ$] every $\psi$-invariant probability $\mu$ satisfies a Ruelle-like inequality,
% \end{itemize}
% among other features (we refer the reader to Section~\ref{statements} for the details on the topology on the space of impulses and the notion of typical). An alternative and somewhat equivalent formulation, although not as natural from the point of view of applications to natural sciences, would be to consider a fixed impulse $I: D \to X$ and to ask whether the previous properties hold whenever one perturbs the original semiflow $\varphi$. We shall not focus on this approach here.

\medskip
 This paper is organized as follows. In Section~\ref{statements} we recall the necessary concepts from the theory of impulsive semiflows 
 and state the main results. Section~\ref{sec:prelim} contains several technical results, which may be of independent interest. This includes the differentiability of Poincar\'e maps, the concept of uniform hyperbolicity for periodic orbits and that such property is preserved by $C^1$-perturbations of the impulse. 
 %which we could not find in the literature. 
We also prove the persistence of hyperbolic periodic orbits for continuous impulsive semiflows and a Franks' type perturbation lemma. 
Section~\ref{sec:GDTDDD} is devoted to the proof of Theorem~\ref{thmAA}, which encloses the information that hyperbolic periodic orbits are dense in the part of the non-wandering set that intersects the impulsive region, for typical impulsive semiflows determined by impulses $I$ that share the same impulsive and landing regions. Part of the technical ingredients in the proof, traced back to the connecting lemmas of Arnaud and Bonatti-Crovisier \cite{Arn,BC04}, appear at the Appendix~A.  %(these are denoted by $D$ and $\hat D$ in Theorem~\ref{thmAA}). 
The main results are Theorem~\ref{thmA} and Corollary~\ref{corA}, proved in 
Sections~\ref{sec:equivrel} and ~\ref{sec:corA} respectively, which are obtained from Theorem~\ref{thmAA} by studying impulsive semiflows in the same equivalence class (determined by the orbits of points in the landing regions) and relating the non-wandering set of any such equivalent impulsive semiflows.
In Section~\ref{sec:examples} we provide several examples to illustrate the application of our results, including the impulsive semiflows derived from predator-prey models in population dynamics, Anosov flows, geometric Lorenz attractors and impulsive billiard flows.
Finally, in Section~\ref{sec:perspectives} we discuss the future perspectives of this perturbative approach in the framework of  impulsive semiflows.

%%%%%%%
\section{Setting and statements}\label{statements}

\subsection{Impulsive semiflows}

Let $X\in \mathfrak{X}^1(M)$ be a vector field on a compact  Riemannian manifold $M$ and let $\varphi:\mathbb{R}_0^+\times M \to  M$ be the $C^1$-smooth flow generated by it, that is to say that $\varphi$  satisfies $\varphi (0,\cdot)=id$ and
\begin{displaymath}
\varphi(s+t,x)=\varphi(s,\varphi(t,x))
	\quad \text{for all $s,t\in \mathbb{R}_0^+$ and $x\in X$.}
\end{displaymath}
We will often use the notation $\varphi_t(x)=\varphi (t,x)$, for every $t\in \mathbb{R}_0^+$ and $x\in X$.
%The $C^1$-smooth semiflow $\varphi$ acting on a \color{red} compact \color{black} metric space is well known to have a \color{red} compact \color{black} and invariant
%non-wandering set, to admit invariant probabilities and any of these ergodic probabilities $\mu$ satisfy Ruelle's inequality
%$$
%h_\mu(\varphi)
%	\leq \int \; \sum_{\{\lambda_j(\mu)\ge 0\}} \lambda_j(\mu)\; d\mu
%	\le \dim X \cdot \int \log \|D\varphi_1\|\, d\mu.
%$$
%Clearly, Ruelle's inequality for any invariant probabilities is obtained from the latter together with the convexity of the entropy
%map and the ergodic decomposition theorem.

\smallskip
Given a non-empty 
%\color{red} \color{red} compact \color{black} \color{black} 
set  $D\subset M$ we will refer to any $C^1$-smooth map $I:D \to M$ so that $I(D)\cap D=\emptyset$ as an \emph{impulse}. 
Using the impulse one can define a impulsive dynamical system
$(M,\varphi,D,I)$, possibly with discontinuities, as follows.
Let $\tau_1:M\to~(0,+\infty]$ be the first hitting time function to the set $D$ given by
\begin{equation}\label{eqtau1}
    \tau_1(x)=\left\{
\begin{array}{ll}
\inf\left\{t> 0 \colon \varphi_t(x)\in D\right\} ,& \text{if } \varphi_t(x)\in D\text{ for some }t>0;\\
+\infty, & \text{otherwise,}
\end{array}
\right.
\end{equation}
which is known to be upper semicontinuous \cite{C04}. 
The subsequent impulsive times are defined in terms of $\tau_1$ applied to the point obtained after the impulse. 
More precisely, setting
$$
\gamma_x(t)=\varphi_t(x), \qquad \forall \,  0\le t<\tau_{1}(x)  \quad \mbox{and} \quad \gamma_x(\tau_1(x))= I(\varphi_{\tau_{1}(x)}(x)),
$$
the remainder of the {impulsive trajectory} and {impulsive times} $(\tau_n(x))_{n\ge 2}$ 
(possibly finitely many) of the point $x\in M$ are defined recursively by 
\begin{displaymath}
\gamma_x(\tau_{n}(x))=I(\varphi_{\tau_n(x)-\tau_{n-1}(x)}(\gamma_x({\tau_{n-1}(x)})))
\end{displaymath}
and
\begin{displaymath}
\tau_{n+1}(x)=\tau_n(x)+\tau_1(\gamma_x(\tau_n(x))).
\end{displaymath}
In particular, 
\begin{equation}\label{def:eq:impulsive}
\gamma_x(t)=\varphi_{t-\tau_n(x)}(\gamma_x(\tau_n(x)))
\quad\text{whenever} \; \tau_n(x)\le t<\tau_{n+1}(x).
\end{equation}

 \color{black}
 Under the assumption that $I(D)\cap D= \emptyset$ is easy to check that $\sup_{n\ge 1}\,\{\tau_n(x)\}=+\infty$ for every $x\in M$, [\cite{AC14}, Remark 1.1], hence the impulsive trajectories are defined for all positive times. The \emph{impulsive semiflow} $\psi_I$ associated to the impulsive dynamical system $(M,\varphi, D, I)$ is now defined as
\begin{equation}\label{def:eq:impulsive2}
\begin{array}{cccc}
        \psi_I:  &  \mathbb{R}^+_0 \times M & \longrightarrow & M \\
        & (t,x) & \longmapsto & \gamma_x(t), \\
        \end{array}
\end{equation}
where $\gamma_x$ stands for the impulsive trajectory of $x$ determined by $(M,\varphi,D, I)$ and described by ~\eqref{def:eq:impulsive}.  We note that $\psi_I$ is indeed a semiflow  [\cite{B}, Proposition 2.1].

The \emph{orbit of} $x\in M$ by the impulsive semiflow $\psi_I$ is the set $O_{x}=\{\gamma_x(t) \colon t\ge 0\}$. 
Whenever necessary we shall write $O_{I,x}$ and $\gamma_{I,x}$ to emphasize the impulse $I$.

\medskip

Throughout the paper, 
\
except if mentioned otherwise, it will be a standard assumption that $\varphi$ is a $C^1$-flow 
on a compact manifold $M$ of dimension $d\ge 3$ 
\color{black} and $D$ is a codimension one 
%\color{red} \color{red} compact \color{black} \color{black} 
%\color{magenta} and connected 
%\color{black} 
submanifold of $M$, transversal to the flow direction,
such that 
\begin{equation}
    \label{eq:noncompact}\tag{H}
    \overline{D} \cap \text{Sing}(\varphi)=\emptyset,
\end{equation}
where $\overline{D}$ stands for the closure of $D$ and 
$\text{Sing}(\varphi)$ stands for the equilibrium points of $\varphi$. 
While every codimension one compact submanifold $D$ transversal to the flow direction satisfies~\eqref{eq:noncompact},
this more general assumption allows to consider relevant examples as impulsive geometric Lorenz attractors
(cf. Example~\ref{ex:Lorenz}).
In this setting the first hitting time function $\tau_1$, defined by ~\eqref{eqtau1}, is piecewise $C^1$-smooth (Theorem 3.3, \cite{Poincare}).

\subsection{Space of impulses}
In order to state our main results we need to define a topology on the space of impulsive semiflows, parameterized by the impulses.
%We consider impulses which are embeddings of its domain and transversal to the vector field. 
More precisely, if $\varphi$ is a $C^1$-flow generated by a vector field $X$, 
$D\subset M$ is a %\color{red} compact \color{black} 
smooth codimension one submanifold transversal to the flow satisfying ~\eqref{eq:noncompact} and $\text{Emb}^1(D, M)$ denotes the space of $C^1$-embeddings of $D$ on $M$, consider the space
\begin{equation}\label{eq:defimpT}
\mathscr I_D
	:=\big\{ I \in \text{Emb}^1(D, M) \colon \text{dist}_H(I(D),D)>0 \;\text{and}\; I(D) \pitchfork X \big\},  
\end{equation}
where $\text{dist}_H$ stands for the Hausdorff distance. 
In the special case that $D$ is compact, the condition
$\text{dist}_H(I(D),D)>0$ can be rewritten as
$I(D)\cap{D} =\emptyset$.
%\begin{equation}\label{eq:defimpT}
%\mathscr I_D
%	:=\big\{ I \in \text{Emb}^1(D, M) \colon I(D)\cap{D} =\emptyset \;\text{and}\; I(D) \pitchfork X \big\}.  
%\end{equation}
%
Endow the space $\mathscr I_D$ with the distance 
%\begin{align*}
%d_{C^1}(I_1,I_2) \, = \, \max\Big\{\, & \max_{x\in D} d\big(I_1(x),I_2(x)\big), \, \max_{y\in I(D)} d\big(I^{-1}_1(y),I^{-1}_2(y)\big),  \\
%	& \, \max_{x\in D} \big\|DI_1(x)-DI_2(x)\big\|,  \, \max_{y\in I(D)} \big\|DI^{-1}_1(y)-DI^{-1}_2(y)\big\| \Big\}
%\end{align*}
\begin{align}\label{def:distC1}
d_{C^1}(I_1,I_2) \, = \, \max\Big\{\,  \sup_{x\in D} d\big(I_1(x),I_2(x)\big), \, \sup_{x\in D} \big\|DI_1(x)-DI_2(x)\big\| \Big\},
\end{align}
where the expression in the right hand-side of \eqref{def:distC1} makes sense after using parallel transport to identify the corresponding tangent spaces
and $d(\cdot,\cdot)$ stands for the usual metric on $M$.
%  By construction, any impulse in $\mathscr I$ satisfies property $(H_2')$.
%$C^1$-functions $I: D \to X$ so that $I(D)\cap D=\emptyset$, $I(D)$ $(H_2)$ holds with the following norm
As $\varphi$ is a $C^1$-flow and $D\subset M$ 
satisfies ~\eqref{eq:noncompact},
%\color{red} \color{red} compact \color{black} \color{black}, 
if $I\in \mathscr I_D$ there exists $\xi>0$ so that for every distinct, $x_1,x_2\in D\cup I(D)$ one has 
$
\{\varphi_t(x_1) \colon  |t|\le \xi\}\cap \{\varphi_t(x_2) \colon  |t| \le \xi\}=\emptyset.
$
%
%We can also consider the space 
%$$
%\mathscr J
%	=\big\{ I \in C^0(D, X) \colon I(D)\subset D'\big\}
%$$ 
%of continuous impulses endowed with the $C^0$-distance
%$$
%d_{C^0}(I_1,I_2) \, = \, \max_{x\in D} d\big(I_1(x),I_2(x)\big)
%$$
%An impulsive semiflow $\psi_I$
%is often  discontinuous, where the 
Any discontinuity point in the interior of $I(D)$  arises as a point mapped by the original flow $\varphi$ on the boundary of the impulsive region $D$ (see Subsection~\ref{subsec:tau1} for more details).
In this way consider the space of impulses
$$
\mathcal{T}=\Big\{ I \in \text{Emb}^1(D, M) \colon  \sup_{x\in I(D)} \Big|\frac{d\tau_1}{dx}(x)\Big| <+\infty \Big\}.
$$
Throughout this paper we consider the space
of impulses
$$
\mathscr I_D^{\mathcal T}:= \mathscr I_D \cap \mathcal{T}.
$$ 
In Subsection~\ref{subsec:tau1} we give further description of this space of impulses, in terms of the underlying flow $\varphi$ and impulsive region.

\subsection{Statements}
%We are now in a position to state our main results. 
The first result concerns the abundance of hyperbolic periodic orbits 
(cf. Subsection~\ref{subsec:hyp} for the definition)
for $C^1$-typical impulsive flows. 

\begin{maintheorem}\label{thmA}
Let $\varphi$ be a $C^1$-flow generated by  $X\in \mathfrak{X}^1(M)$ and $D$ be a smooth
%, \color{red} \color{red} compact \color{black} \color{black} 
submanifold of codimension one transversal to $X$
satisfying ~\eqref{eq:noncompact}.
There exists a Baire residual subset 
 $\mathscr R
\subset \mathscr I_D^{\mathcal T}$ of impulses
%$$
%\mathscr I_D
	%=\big\{ I \in \text{Emb}^1(D, M) \colon %I(D)\cap{D} =\emptyset \;\text{and}\; I(D) %\pitchfork X \big\}
	%$$
such that the impulsive semiflow $\psi_I$ 
determined by $I \in \mathscr R$ satisfies
\begin{equation}
\label{eq:thmAeq}
%\overline{Per_h(\psi_I) \cap I(D)} = %{\Omega(\psi_I) \cap I( D)} 
\overline{Per_h(\psi_I)} \cap D = {\Omega(\psi_I) \cap D}
    \end{equation} 
where ${Per_h(\psi_I)}$ denotes the set of hyperbolic periodic orbits of $\psi_{I}$. 
\end{maintheorem}

%\begin{remark}
Some comments are in order. 
Firstly, as $\psi_{I,t}(x)=\varphi_t(x)$ for every $x\in D$ and for every small $t>0$, periodic orbits under $\psi_I$ never intersect $D$. In this way,  if $\varphi$ has no periodic orbits in $D$ then $\text{Per}_h(\psi_I) \cap D=\emptyset$. Secondly, the conclusion of Theorem~\ref{thmA} cannot be written using the landing region $I(D)$, as there exist 
%However, equation ~\eqref{eq:thmAeq} in the conclusion of Theorem~\ref{thmA} does not imply a similar expression on the range $I(D)$ of the impulse. More precisely, there are examples of 
$C^1$-open sets of impulses for which the equality
$\overline{Per_h(\psi_I) \cap I(D)} = {\Omega(\psi_I) \cap I( D)}$
fails (see Example~\ref{ex:TBA}).
%\end{remark}

\medskip
The previous theorem allows to describe, for typical impulsive semiflows, the portion of points in the non-wandering set whose forward 
orbit intersects the impulsive region. The following result gives sufficient conditions for the invariance of the whole non-wandering set.

\begin{maincorollary}\label{corA}
Let $\varphi$ be a $C^1$-flow generated by  $X\in \mathfrak{X}^1(M)$ and $D$ be a smooth
submanifold of codimension one transversal to $X$
satisfying ~\eqref{eq:noncompact}.
%Assume that either:
%\begin{itemize}
%%\item[(i)] $\partial D=\emptyset$
%\item[(i)] $\Omega(\varphi) \cap \partial( I(D))=\emptyset$, or
%\item[(ii)] $\varphi$ is minimal.
%\end{itemize}
%There exists a Baire residual subset $\mathscr R\subset \mathscr I_D$ so that, for every $I\in \mathscr R$ the non-wandering set $\Omega(\psi_I)$ is a \color{red} compact \color{black} and $\psi_I$-invariant
%subset of $M$.
The following hold:
\begin{enumerate}
\item 
if $I_0\in  \mathscr I_D^{\mathcal T}$ is such that
  %$\Omega(\varphi) \cap \partial( I_0(D))=\emptyset$
  $\Omega(\psi_{I_0}) \cap \partial D =\emptyset$ 
 then there exist $\delta>0$,
an open neighborhood $\mathcal V$ of $I_0$ and  a Baire residual subset 
$\mathscr R\subset \mathcal V$ so that, for every   $I\in \mathscr R$
one can write the non-wandering set $\Omega(\psi_I)$ as a (possibly non-disjoint) union
%$$
%\Omega(\psi_I) =  \overline{Per_h(\psi_I)} \cup \, \Omega_2(I,%\varphi),
%$$
$$
\Omega(\psi_I) =  \overline{Per_h(\psi_I)} \cup \, \Omega_2(\varphi, D),
$$
where %$\Omega_2(I, \varphi)\subset \Omega(\psi_I)$ 
$\Omega_2(\varphi, D)\subset \Omega(\psi_I)$ 
 is a $\varphi$-invariant set 
which does not intersect a $\delta$-neighborhood of the cross-section $D$. 
Moreover, the set $\Omega(\psi_I)\setminus D$ is a $\psi_I$-invariant subset of $M$.
%
%
%the set $\Omega(\psi_I)\setminus D$ \color{black} is a $\psi_I$-invariant subset of $M$.
\item if $\varphi$ is minimal then there exists a Baire residual subset $\mathscr R\subset \mathscr I^{\mathcal T}_D$  so that, for every $I\in \mathscr R$, the set of hyperbolic periodic orbits is dense in $\Omega(\psi_I)$.
Moreover, the set $\Omega(\psi_I)\setminus D$ is a $\psi_I$-invariant subset of $M$.
\end{enumerate}
\end{maincorollary}
\color{black}

\medskip

%Some comments are in order. 
%Observe that the statement in Theorem \ref{thmA} %concerns periodic orbits that intersect the %impulsive set, 
%due to the fact that the non-wandering set can %have more than one connected component, some of %which may not intersect the image of the %impulsive set. %(see Example~\ref{ex:bonitinho}).
We observe that, in the special case that $\partial D=\emptyset$, Corollary~\ref{corA} ensures that the impulsive semiflows $\psi_I$ for which $\Omega(\psi_I)\setminus D$
%the non-wandering sets 
fails to be invariant 
and that admit no invariant probability measure are not generic (compare to the results
in \cite{AC14}).  
%Furthermore, if condition $(ii)$ holds then all orbits of the  original flow $\varphi$ intersect $D$  and for that reason we can guarantee that the periodic points are generically  dense. 
%Finally, 
Example~\ref{ex:ACrobusto} illustrates that one cannot even expect the full invariance of the non-wandering set $\Omega(\psi_I)$ for a Baire generic set of impulses $I$.

\medskip

To finalize this section we shall make some comments about the strategy used in the proof of Theorem~\ref{thmA}. A more specific context would
be to consider a family of impulses which have a common domain and range. Indeed, if $D, \hat D$ are smooth codimension one submanifolds of $M$ transversal to
the flow and so that $D\cap \hat{D}=\emptyset$ consider the space of impulses
%$$
%\mathscr I_{D,\hat D}
%	=\Big\{ \text{$C^1$-diffeomorphisms} \; I : D %\to I(D)\subset \hat D \Big\}
%$$
\begin{align}
\mathscr I^{\mathcal T}_{D,\hat D}
	& = \mathscr I^{\mathcal T}_{D} \cap 
 \Big\{ \text{$C^1$-diffeomorphisms} \; I : D \to \hat D \Big\} \nonumber
 \\
 &
 = \Big\{ I \in \text{Diff}^{\,1}\!(D,\hat D)  \colon   \sup_{x\in \hat D} \Big|\frac{d\tau_1}{dx}(x)\Big| <+\infty \Big\}
 \label{eq:DDtau1}
\end{align}
endowed with  the $C^1$-topology defined by ~\eqref{def:distC1}. 
Condition $\sup_{x\in \hat D} \Big|\frac{d\tau_1}{dx}(x)\Big| <+\infty$ depends only on the initial flow $\varphi$ and on the submanifolds $D,\hat D$. As a consequence, the metric space $\mathscr I^{\mathcal T}_{D,\hat D}$ is either empty or the whole space $\text{Diff}^{\,1}\!(D,\hat D)$, hence complete. In this special context there are several simplifications of the problem of obtaining denseness of periodic orbits for typical impulsive semiflows (e.g. the impulsive time $\tau_1$ does not vary with perturbations of the impulses $I$), allowing to focus on the major technical difficulties of obtaining a general density theorem for impulsive flows that may have discontinuities. 
The next theorem is the main technical result of this paper.   

\begin{maintheorem}\label{thmAA}
Let $\varphi$ be a $C^1$-flow generated by  $X\in \mathfrak{X}^1(M)$ and
let $D, \hat D$ be smooth codimension one submanifolds of $M$ transversal to
the flow so that ~\eqref{eq:noncompact} holds and 
$\overline D\cap \overline{\hat{D}}=\emptyset$.
There exists a Baire residual $\mathscr{R}$ subset of the set of impulsive maps 
$\mathscr I^{\mathcal T}_{D,\hat D}$
such that the impulsive semiflow $\psi_I$ 
determined by  $I \in \mathscr{R}$ 
satisfies:
\begin{equation}\label{eq:dchapeu}
%\overline{Per_h(\psi_I) \cap \hat D} = {\Omega(\psi_I) \cap \hat D} 
\overline{Per_h(\psi_I) }\cap D = \Omega(\psi_I) \cap D.
\end{equation}
%
%
%There exists a Baire residual subset $\mathscr R \subset \mathscr I_{D,\hat D}$  
%such that the impulsive semiflow $\psi_I$ 
%determined by $I \in \mathscr R$ satisfies:
%\begin{enumerate}
%\item the non-wandering set $\Omega(\psi_I)$ is \color{red} compact \color{black} and $\psi_I$-invariant;
%\item $\Omega(\psi_I)= \overline{Per_h(\psi_I)}$, where ${Per_h(\psi_I)}$ denotes the set of hyperbolic periodic orbits of $\psi$. 
%%is dense in $\Omega_\psi$
%\end{enumerate}
%% so that $$\Omega(\psi_I)= \overline{Per_h(\psi_I)}  \quad \mbox{for all} \ I\in \mathcal{R}. $$
%%In particular,  the non-wondering set $\Omega(\psi_I)$ is invariant  for all $I \in \mathcal{R}$.
\end{maintheorem}

%We note that the conclusion of the previous theorem is stated in terms of the local cross-section $\hat D$. This is due to the fact that if $x\in D$ and $t>0$ is small then $\psi_{I,t}(x)=\varphi_t(x)$, hence periodic orbits do not intersect $D$. In any case, we obtain
% that \eqref{eq:dchapeu} implies 
%$
%\color{red}
%\overline{Per_h(\psi_I) }\cap D = \Omega(\psi_I) \cap D.
%\color{black}
%$ 
%\medskip

While there exists no \emph{a priori}  relation between the statement of Theorem~\ref{thmA} and Theorem~\ref{thmAA} 
%Nevertheless, it is interesting to notice that Theorem~\ref{thmA} will be 
the first one will be obtained as consequence of the later, once one introduces equivalence classes in the space of impulsive semiflows and relate their non-wandering sets
(cf. Section~\ref{sec:equivrel}).

%Consider $C^1$-smooth 
%semiflow $\varphi:\mathbb R_0^+\times  X\to X$ and submanifolds $D,D'\subset X$ satisfying properties $(H_1)$ and
%$(H_3)-(H_5)$ above and 
%there exists $\xi>0$ so that:
%\begin{enumerate}
%\item[$(H_2')$] (local injectivity) $\{\varphi_t(x_1) \colon  |t|\le \xi\}\cap \{\varphi_t(x_2) \colon  |t| \le \xi\}=\emptyset$ for every distinct $x_1,x_2\in D\cup D'$.
%\end{enumerate}
%
%%
%We endow the space 
%$$
%\mathscr I
%	=\big\{ I \in C^1(D, M) \colon I(D)\subset \hat{D}\big\}
%$$ 
%of impulses with the distance
%$$
%d_{C^1}(I_1,I_2) \, = \, \max\Big\{\, \max_{x\in D} d\big(I_1(x),I_2(x)\big), \, \max_{x\in D} \big\|DI_1(x)-DI_2(x)\big\| \Big\}.
%$$
%  By construction, any impulse in $\mathscr I$ satisfies property $(H_2')$.
%%$C^1$-functions $I: D \to X$ so that $I(D)\cap D=\emptyset$, $I(D)$ $(H_2)$ holds with the following norm
%In the special case that the original semiflow $\varphi$ is invertible, meaning it is a $C^1$-flow, assumption $(H_2)$ is automatically satisfied.
%%
%We can also consider the space 
%$$
%\mathscr J
%	=\big\{ I \in C^0(D, X) \colon I(D)\subset D'\big\}
%$$ 
%of continuous impulses endowed with the $C^0$-distance
%$$
%d_{C^0}(I_1,I_2) \, = \, \max_{x\in D} d\big(I_1(x),I_2(x)\big)
%$$

\begin{remark}
We note that the assumption that $\varphi$ is a flow is not used in full strength, as it is only used that $\varphi$ is injective in local open neighborhoods of the domain %and image
of the impulse functions. Moreover, it seems plausible that a version of Theorem \ref{thmA} can be proven for certain classes of non-invertible impulsive maps (the invertibility assumption is only strongly used for the invariance of the non-wandering sets in Corollary \ref{corA}). We will not use nor prove this fact here. 
\end{remark}

\section{Preliminaries}\label{sec:prelim}

%%%%%%%%%
\subsection{Discontinuity sets for impulsive semiflows}\label{subsec:tau1}

Let $\varphi$ be a $C^1$-flow on the manifold $M$, let $I: D \to M$ be a $C^1$-impulse and $\psi_I$ be the impulsive semiflow.
There are at least two sources of difficulties in the analysis of impulsive semiflows, namely the discontinuities of $\psi_I$ in $I(D)$ (determined by the boundary of $D$) and the geometry of the discontinuity set (which determines the cardinality of connected components of the points that return to $D$).
%$A_{\alpha}$
%that are determined by points in the range of %the impulsive map which are non-recurrent (as %singularities), or the number of connected %components.
Let us be more precise.
In case of a boundaryless submanifold $D$ one has a simple decomposition
$$
I(D) = \{x\in I(D)\colon \tau_1(x)< +\infty\}
\; \cup \; \{x\in I(D)\colon \tau_1(x)=+\infty\}.
$$
%where the first set is open and the second set may be non-empty. THIS WAS FALSE
The second set in the right-hand side above may be responsible by the failure of the estimate in 
\eqref{eq:DDtau1}. Indeed, there are examples where $\{x\in I(D)\colon \tau_1(x)=+\infty\}$ has empty interior and consists of wandering points for the impulsive semiflow $\psi_I$ but still
$$
\sup_{x\in I(D)} \Big|\frac{d\tau_1}{dx}(x)\Big| =+\infty
$$
(cf. Example~\ref{ex:Lorenz} and Remark~\ref{rmkLorenz}).

In general, using that the discontinuity points of the impulsive semiflow $\psi_I$ defined by ~\eqref{def:eq:impulsive2} in the interior of $I(D)$ arise as points whose orbits by the original flow $\varphi$ intersect  the boundary of the impulsive region $D$, %there exists an at most countable set $A\subset \mathbb N_0$ so that 
the space 
$$
\text{interior}(I(D)) \setminus \{x\in I(D)\colon \tau_1(x)=+\infty\}
$$ 
 %it is formed by 
    %points whose forward iterate always remains 
    %in the complement of D, which is open
is covered by a collection of closed connected components
$(V_\alpha)_{\alpha\in A}$ such that:
\begin{itemize}
    \item[(i)] $V_\alpha\cap V_\beta \subset \partial V_\alpha$;
    \item[(ii)] $\partial V_\alpha 
        \; \subset\; \bigcup_{t>0} \,\varphi_{-t}(\partial D)$; 
    \item[(iii)] $\tau_1$ is $C^1$-smooth on the interior of ${V_\alpha}$ 
\end{itemize}
for every $\alpha,\beta\in A$. Indeed, if $x\in \text{interior}(I(D))$ and $\tau_1(x)<\infty$
then either there exists $t>0$ so that $\varphi_{t}(x)\in \text{interior}(D)$, hence it belongs to the interior of a closed and connected set $V_\alpha$ such that 
$\partial V_\alpha 
        \; \subset\; \bigcup_{t>0} \,\varphi_{-t}(\partial D)$,
or there exists $t>0$ so that $\varphi_{t}(x)\in \partial D$, in which case $x$ belongs to the boundary of such a closed connected set $V_\alpha$. Since $I(D)$ %is \color{red} compact \color{black}, hence 
is separable, there are at most countably many connected components $V_\alpha$.

\medskip
In the special case that $I\in \mathscr I^{\mathcal T}_{D,\hat D}$, the at most countable connected components $(V_\alpha)_{\alpha\in A}$ are  determined by $\varphi$, $D$ and $\hat D$, hence independent of the impulse $I$. Similarly, the condition

%While the set $\mathscr I^{\mathcal T}_{D}$ is hard to describe, in the proof of Theorem~\ref{thmA} we consider equivalence classes for impulses whose ranges determine the same space of orbits. That is to say, that the sets  
%$\mathscr I^{\mathcal T}_{D,\hat D}$ are the ones that are strongly used in our arguments. Indeed, for each pair of fixed submanifolds $D,\hat D$, 
%an element 
%$I\in \mathscr I^{\mathcal T}_{D,\hat D}$ is a $C^1$-diffeomorphism $I : D \to \hat D$ such that 
\begin{equation}
    \label{eqhyponD}
\sup_{x\in \hat D} \Big|\frac{d\tau_1}{dx}(x)\Big| <+\infty,
\end{equation}
depends intrinsically on $\varphi$, $D$ and $\hat D$ but not on the impulse $I$.

%The following lemma provides a number of sufficient conditions for this property to be satisfied.

%\color{bgreen}
%\begin{lemma}\label{hypothesison}
%\marginpar{\tiny \color{red} faltam condicoes suficientes}
%    Let $\varphi$ be a $C^1$-flow generated by  $X\in \mathfrak{X}^1(M)$, $D,\hat D$ be smooth submanifolds of codimension one transversal to $X$.
 %   If either
  %  \begin{enumerate}
   %     \item ...
    %    \item ..., or 
     %   \item ...
   % \end{enumerate}
   % then condition ~\eqref{eqhyponD} holds.
%\end{lemma}

%%%%%%%%%
\subsection{Poincar\'e maps for impulsive semiflows}\label{subsec:Poincare}

Given a smooth flow, a periodic orbit $\gamma$ and a local cross-section $\Sigma$ passing through a point $p \in \gamma$, the Poincar\'e map
$P: U\subset \Sigma \to \Sigma$ is a local diffeomorphism defined on some open neighborhood $U\subset \Sigma$ of $p$ which is 
the first return map to $\Sigma$ of the points in $\mathcal U$.  
%This map is  called  the Poincar\'e map and it a diffeomorphism. 

In case of impulsive semiflows, by some abuse of notation we shall denote as Poincar\'e map the ones that concatenates the impulse and the first hitting time map to $D$. More precisely, given $I\in \mathscr I_D$ we shall consider the map (which we denote as a Poincar\'e map)
\begin{equation}\label{defPo00}
\begin{array}{cccc}
P_I : & \widehat{I(D)} & \to & I(D) \\
	& x & \mapsto & I \circ \varphi_{\tau(x)}(x)
\end{array}
\end{equation}
where 
\begin{equation}\label{deftaup}
\tau(x)=\inf\{t>0 \colon \varphi_t(x)\in D\} 
\end{equation} 
for every $x\in I(D)$ and 
$\widehat{I(D)}=\{ x\in I(D) \colon \tau(x) <\infty\}.$
Notice that $\tau=\tau_1\mid_{I(D)}$.
Consider the  set 
$$
\widehat{I_*(D)}=\{x\in \widehat{I(D)} \colon \varphi_{\tau(x)}(x) \notin \partial D \}
$$
which consists of the points in $\widehat{I(D)}$ which are mapped in the interior of $D$.
%
%\begin{figure}[htb]\label{fig1}
%\begin{center}
%  \includegraphics[width=10cm,height=5cm]{pic1.png}
%  \caption{Poincar\'e map of an impulsive semiflow}
%\label{figure}
%\end{center}
%\end{figure}
Let $\mathscr I_{D,\hat D}$ be the space of $C^1$-diffeomorphisms from $D$ to $\hat D$. 
In what follows we prove that, even though the Poincar\'e maps in \eqref{defPo00} are defined by expressions which take into account the impulse, these keep being $C^1$-piecewise smooth. 

%
%\marginpar{\tiny incluir a prova}
%\begin{proposition}\label{prop:smoothnessP}
%Let $\varphi$ be a $C^1$-flow generated by  $X\in \mathfrak{X}^1(M)$, $D,\hat D$ be smooth submanifolds of codimension one transversal to $X$ and $I \in \mathscr{I}_D$ be an  impulse. Let $\gamma$ be a periodic orbit of the impulsive semiflow $\psi_I$ such that 
%there exists $\hat p  \in (\gamma \cap \hat{D}) \setminus \partial \hat{D}$.
%%$p\in \gamma \cap D$,  $\hat p = I(p) \in (\gamma \cap \hat{D}) \setminus \partial \hat{D}$. 
%Then there exist a local cross-section $\Sigma\subset \hat D$,  an neighborhood  $V\subset \Sigma \subset \hat{D}$ of $\hat{p}$ and a $C^1$-function $\tau: V\to \mathbb R_+$ such that the Poincar\'e map defined by
%%$P_I: V \to \Sigma$ 
%\begin{equation}\label{defPo}
%\begin{array}{cccc}
%P_I : & V & \to & P_I(V)\subset \Sigma \\
%	& x & \mapsto & I \circ \varphi_{\tau(x)}(x)
%\end{array}
%\end{equation}
%is a $C^1$-diffeomorphism.  
%\end{proposition}

\begin{proposition}\label{prop:smoothnessP}
Let $\varphi$ be a $C^1$-flow generated by  $X\in \mathfrak{X}^1(M)$, $D,\hat D$ be smooth submanifolds of codimension one transversal to $X$ so that 
$\overline{D}\cap \overline{\hat {D}}=\emptyset$, and let $I \in \mathscr I_{D,\hat D}$ be an  impulse. Given $x \in \widehat{I_*(D)}$
%$p\in \gamma \cap D$,  $\hat p = I(p) \in (\gamma \cap \hat{D}) \setminus \partial \hat{D}$. 
there exist a local cross-section $\hat{\Sigma}\subset \hat D$ and  a neighborhood  $\hat{V}\subset \hat{\Sigma}$ of $x$  such that the Poincar\'e map $P{_I}|_{\hat{V}} $ is a $C^1$-diffeomorphism.  
\end{proposition}
\begin{proof}
%Suppose that $\bar{\tau}$ is the period of  the periodic orbit $\gamma$. Given $p\in \gamma \cap D$, 
%Since $D$ and $ \hat{D}$ are both \color{red} compact \color{black} submanifolds such that  $D\cap \hat{D}= \emptyset$  
By the assumption $\overline{D}\cap \overline{\hat {D}}=\emptyset$,
there is an open neighborhood $W$ of $\hat{D}$ such that  $\psi_I |_{W} = \varphi$. 
Let $x \in \widehat{I_*(D)}$ and consider a local cross-section  $\hat{\Sigma}\subset \hat D$. %By  transversality, 
The flowbox theorem for $\varphi$ guarantees the existence of  a flow-box $(V_1, h_1)$ at $x$  where $V_1 \subset W $ is an open neighborhood of $x$,  
and the chart map $h_1: V_1\to [0,1]\times U_1 \subset[0,1]^m$ sends trajectories of $X$ in $V_1$  into trajectories of the vector field 
$X_V(z):= (1,0, \cdots, 0)$ for every $z\in [0,1]^m$, where $m=\dim D$.

Choose $t_0>0$ small such that $\varphi_{t_0}(x)$ still lies in $V_1$  
and consider the compact segment of orbit  $\Gamma:= \{ \varphi_t(x):  t_0 \le  t \le  \tau(x)\}$. Since   $\tau(x)< \infty$,  %by the long flow-box theorem (see Proposition~1.1 in \cite{WPalis}) 
there exists a long tubular flow-box chart $(V_2, h_2)$ for the original flow $\varphi$, say
$h_2: V_2\to [0,1]\times U_2 \subset[0,1]^m$
 such that $V_2 \supset \Gamma$
(see Proposition~1.1 in \cite{WPalis}). 
Let $\Sigma_1=h_2^{-1}(\{0\} \times U_2)$ and $\Sigma_2=h_2^{-1}(\{1\} \times U_2)$ be the components of the boundary of $V_2$ which are transversal to the flow $\varphi$.  Denote by $\pi_1: \hat{V} \subset \hat{\Sigma}\to \Sigma_1$,  $\pi_2: \Sigma_1 \to  \Sigma_2$,  the projections along the positive trajectories of $\varphi$ given by the tubular flowbox theorems, where $\hat{V}$ is a small open neighborhood of $x$ in $\hat{\Sigma}\cap \widehat{I_*(D)}$. Similarly, let  $\pi_3: \Sigma_2 \to  \Sigma$ be the projection along the negative trajectories of $\varphi$, where $\Sigma$ is a small neighborhood of $\varphi_{\tau(x)}(x)$ in $D$. 
Each $\pi_i$, $i=1,\cdots, 3$ is a $C^1$ map (cf. Proposition~1.1 in \cite{WPalis}). 
Since
$$
P{_I}|_{\hat{V}}\equiv I\circ \pi_3\circ \pi_2\circ\pi_1
$$ 
and the impulse $I$ is $C^1$-smooth we have that $P{_I}|_{\hat{V}}$ is $C^1$. 

Now, using that $I$ is invertible and $\varphi$ is a flow the same argument can be done for construct the inverse of the Poincaré map as before, proving that its inverse is also $C^1$. 
%Since $DI_p$ is an isomorphism, we have that $I$ is a local diffeomorphism on some open neighborhood of $p$ inside the cross-section. Suppose $I: \tilde{U_0} \to I(\tilde{U_0})$ is a diffeomorphism such that  $U_0= I(\tilde{U}_0) \subset V\cap \hat{D}$ is a transversal section.
%
%Let $W_0 := \{w\in W; \exists \lambda \in[0,1] \mbox{such that} \ (\lambda, w)\in h(U_0) \}$ be the projection of the set $h(U_0)$ onto $\mathbb{R}^{m-1}$. Note that $[0,1]\times W_0\subset [0,1]\times W$ contains $h(p)=0$. We set the following
%\begin{eqnarray*}
%  V_0&:=& h^{-1}([0,1]\times W_0)  \\
%  \Sigma_0&:=&  \varphi_{\bar{\tau}}^{-1}(V_0)\cap I(U_0)\\
%  \Sigma_2&:=& \varphi _{\bar{\tau}}(\Sigma_0)
%\end{eqnarray*}
%Note that $V_0\subset V$, $\Sigma_0 \neq \varnothing$, because $I(p)\in \Sigma_0$ and $\Sigma\subset V_0$ contains the points that after the impulse return to $V_0$ in time $\bar{\tau}$.
%\begin{figure}[htb]
%\begin{center}
%  \includegraphics[width=8cm,height=4cm]{poincare.jpg}
%  \caption{The flowbox}
%\label{figure}
%\end{center}
%\end{figure}
%Moreover, $\Sigma_2$ is a transversal section at $p$ since $\Sigma_0$ is a locally transversal section at $I(p)$ and the flow map $\varphi_{\bar{\tau}}$ is smooth on $\Sigma_0.$
\end{proof}

\begin{remark}
%A version of Proposition \ref{prop:smoothnessP} was obtained in  for certain 
We observe that Poincar\'e maps different from ours, where the impulse is prior to the action by the flow map,
were defined in \cite{Poincare}. 
%Moreover, we do not assume the existence of a periodic orbit. 
Moreover, there is no need for invertibility of the impulse in the proof of the $C^1$-differentiability of the Poincar\'e map in 
Proposition \ref{prop:smoothnessP}.
\end{remark}

%As our context of impulsive semiflows have a flow as a starting point then we have the following strong property, which ensures 
We now show the invariance of the non-wandering set before hitting the impulsive region.

\begin{lemma}\label{le:invOmega}
Let $\varphi$ be a $C^1$-flow generated by  $X\in \mathfrak{X}^1(M)$, $D,\hat D$ be smooth submanifolds of codimension one transversal to $X$ and $I \in \mathscr{I}_{D,\hat{D}}$ be an  impulse. If $x\in \Omega(\psi_I) \setminus D$ then $\gamma_x(t)\in \Omega(\psi_I)$
for every $0\le t < \tau_1(x)$.
\end{lemma}

\begin{proof}
Let $x\in \Omega(\psi_I) \setminus D$ and $0< t < \tau_1(x)$ be arbitrary. As $x\notin D$ %and $0< t < \tau_1(x)$ 
then $\gamma_x(s)=\varphi_s(x)$
for every $0< s < \tau_1(x)$. Hence we need to prove that $\varphi_t(x)\in \Omega(\psi_I)$.
If $U$ is an open neighborhood of $\varphi_t(x)$ then $W=\varphi_{-t}(U)$ is an open neighborhood of $x$ (we may assume without loss loss of generality that $U\cap D=\emptyset=\varphi_{-t}(U)\cap D$). Therefore, there exist $s>0$ and $y\in W$ so that $\gamma_y(s) \in W$.
In particular $z=\varphi_t(y) \in U$ is so that 
$$
\gamma_z(s)= \gamma_{\varphi_t(y)}(s) = \gamma_{y}(t+s)
=\varphi_t(\gamma_y(s))  \in U.
$$
As $U$ is arbitrary, this shows that $\varphi_t(x)$ is a non-wandering point, and this proves the lemma.
\end{proof}

%\begin{remark}
%The argument in the proof of the previous lemma can be used to ensure that $I^{-1}(p) \in \Omega(\psi_I) \cap \text{interior}({D})$
%whenever $p \in \Omega(\psi_I) \cap \text{interior}(\hat{D})$.  
%\end{remark}

%%%%%%%%%%%%%%%
\subsection{Hyperbolicity and permanence of periodic orbits}\label{subsec:hyp}
%\section{Hyperbolicity}

In this subsection we introduce %recall the 
a concept of hyperbolicity for periodic orbit of impulsive semiflows and prove their stability. Even though the following notions are quite natural and probably known, we could not find them in the literature.

Recall that $\gamma: \mathbb R_+ \to M$ is a \emph{periodic orbit} for the impulsive semiflow if 
the orbit $\gamma=\{\gamma_x(t) \colon t\ge 0\}$ of  $x\in M$ satisfies  
$\gamma_x(T)=x$, for some $T>0$ (recall expressions \eqref{def:eq:impulsive} and ~\eqref{def:eq:impulsive2}). Note that periodic orbits $\gamma$ that intersect both $D$ and $\hat D$ are never compact, as $I^{-1}(\gamma \cap \hat D) \cap \gamma=\emptyset$.
The $C^1$-smoothness of the Poincar\'e maps in \eqref{defPo00} are essential to define the notion of hyperbolicity for periodic orbits that do not intersect the boundary of the impulsive sets, as follows.

\begin{definition}\label{defhyp}
Let $\varphi$ be a $C^1$-flow generated by  $X\in \mathfrak{X}^1(M)$, $D,\hat D$ be smooth submanifolds of codimension one transversal to $X$, $I\in \mathscr I_{D,\hat D}$ and $\gamma$ be a periodic orbit for the impulsive semiflow $\psi_I$
so that $\gamma\cap \partial \hat D=\emptyset$.
We say that $\gamma$ is a \emph{hyperbolic periodic orbit} if there exist points $x_1, x_2, \dots, x_N \in \gamma \cap \hat D$ and open neighborhoods
$V_1, V_2, \dots, V_N \subset \hat D$ so that:
\begin{enumerate}
\item the Poincar\'e maps $P_I\mid_{V_i}:  V_i \to P_I(V_i)\subset \hat D$ are $C^1$-smooth
\item $x_{i+1}=P_I(x_i)$ for every $1 \le i \le N-1$ and $x_{1}= P_I(x_N)$;
\item $x_1$ is a hyperbolic fixed point for the $C^1$-map $f=P_I^N$, restricted to a small open neighborhood of $x_1$.
\end{enumerate} 
\end{definition}

\begin{remark}
For periodic orbits whose  closure 
does not intersect the impulsive set, 
the notion of hyperbolicity is the classical one. 
Furthermore, such periodic orbits are compact, they will not intersect the impulsive sets for any $C^1$-small perturbation of the impulse. %Finally,  we shall not need to consider periodic orbits that intersect the boundary of the impulsive sets.
\end{remark}

The following lemma guarantees that hyperbolicity is persistent by arbitrary small $C^1$-perturbations of the impulse. As the domain of the impulses is $D$ we can only consider periodic orbits $\gamma$ so that $\overline{\gamma}\cap D\neq \emptyset$.

\begin{lemma}\label{cont.hip}
Let $\varphi$ be a $C^1$-flow generated by  $X\in \mathfrak{X}^1(M)$, $D,\hat D$ be smooth submanifolds of codimension one transversal to $X$, $I\in \mathscr I_{D,\hat D}$ and $\gamma$ be a periodic orbit for the impulsive semiflow $\psi_I$
so that $\overline{\gamma}\cap \partial D=\emptyset$.
If $\gamma$ is a hyperbolic periodic orbit of period $T>0$ then:
\begin{enumerate}
\item there exists an open neighborhood $\mathcal V$ of $\gamma$ such that every periodic orbit $\hat\gamma$ for $\psi_I$ which is contained in $\mathcal V$ is hyperbolic; 
\item there exists a $C^1$-open neighborhood $\mathcal A$ of the impulse $I$, an open neighborhood $\mathcal V$ of $\gamma$ and $\varepsilon>0$
so that for every $J\in \mathcal{A}$ there exists a unique hyperbolic periodic orbit $\gamma_J$ for the impulsive semiflow $\psi_J$, which is contained in 
$\mathcal V$ and whose period belongs to the interval $(T-\varepsilon, T+\varepsilon)$.
\end{enumerate}
\end{lemma}

\begin{proof}
As the stability and robustness of hyperbolic periodic orbits is well known for smooth dynamical systems \cite{Shub}, if $\overline \gamma \cap D=\emptyset$ there is nothing to prove. 
Assuming otherwise, we shall emphasize the main differences for impulsive dynamical systems.
Assume that $\gamma$ is a hyperbolic periodic orbit as in the statement of the lemma, and 
consider the points $x_1, x_2, \dots, x_N \in \gamma \cap \text{interior}(\hat D)$, the open neighborhoods
$V_1, V_2, \dots, V_N \subset \hat D$ and the $C^1$ Poincar\'e maps $P_I\mid_{V_i}: V_i \to P_I(V_i)\subset \hat D$ 
given according to Definition~\ref{defhyp}. 
Then, there exists a small open neighborhood 
$V\subset V_1$ where the map $f=P_I^N$ is a well defined $C^1$-map
and $x_1$ is a hyperbolic fixed point for $f$.

Given $\delta>0$ we may diminish the open
neighborhoods $(V_i)_{1\leq i \leq N}$, if necessary, in such a way that these become pairwise disjoint and 
$ 
\tau(x) \in [\tau(x_i)-\delta, \tau(x_i)+\delta]
$ 
for every $x\in V_i$ and $1\le i \le N$.
%We may assume, without loss of generality, that the collection of neighborhoods $(V_i)_{1\leq i \leq N}$ is pairwise disjoint. In particular
Every periodic orbit for the impulsive semiflow contained in the 
open neighborhood 
\begin{equation}\label{eqnhoddVV}
\mathcal V := \bigcup_{i=1}^{N-1} \, \Big\{ \varphi_t(x) \colon 0\le t \le \tau(x) \; \& \; x\in V_i \Big\}
\end{equation}
of $\gamma$ induces a periodic orbit for $f$ in $V$.
 %, by the classical characterization of uniform, one can reduce the  
 Reducing $V$ if necessary, using that $x_1$ is a hyperbolic fixed point for $f$, there exists a family of stable and unstable cone fields $(C^s(x))_{x\in V}$ and $(C^u(x))_{x\in V}$ preserved by 
 $Df^{-1}$ and $Df$, respectively (see \cite{Shub} for more details).
Now, every periodic orbit $\tilde \gamma$ in the neighborhood $\mathcal V$ defined by \eqref{eqnhoddVV} induces a periodic point $\tilde x_1\in V$ of some period $\ell\ge 1$ for $f$ which, 
due to the existence of the family of stable and unstable cone fields, is itself hyperbolic. 
%We claim that there exists an open neighborhood $\mathcal V$ of $\gamma$ such that every %periodic orbit $\hat\gamma$ for $\psi_I$ which is contained in $\mathcal V$ is hyperbolic. 
This finishes the proof of item (1).

The proof of item (2) is a simple adaptation of the classical argument for continuation of elementary periodic points. 
If $\gamma$ is a hyperbolic periodic orbit for 
$\psi_I$ as above, there exists an open neighborhood $V\subset \hat D$ of $x_1$  so that 
$
f=P_I^N
$
is a $C^1$-diffeomorphism from $V$ onto its image.
We may assume that $V\subset \mathbb R^s$, where $s=\dim \hat D$ (diminishing $V$ if necessary and considering a $C^1$-chart 
$\psi: V \to \psi(V)\subset \mathbb R^s$). The map
$$
\begin{array}{cccc}
F :    & V \times C^1(V,\mathbb R^s) & \to & \mathbb R \\
	& (x,g) & \mapsto & g(x)-x
\end{array}
$$
is such that (i) $F(x_1,f)=0$, and (ii) $D_x F(x_1,f)=Df(x_1)-Id$ is an isomorphism (by hyperbolicity of $x_1$). 
In consequence, the implicit function theorem guarantees that there exists a $C^1$-open neighborhood $\mathcal U\subset C^1(V,\mathbb R^s)$
of $f$ and a $C^1$-function $p: \mathcal U \to V$ so that $F(p(g),g)=0$ for every $g\in \mathcal U$. In particular, 
as $Dg$ is $C^0$-close to $Df$, it follows that $p(g)$ is the unique fixed point for the diffeomorphism $g$ on $V$, and it is hyperbolic. 

We now make explicit the dependence of the map $f$ on the impulsive map.
As the function $J\mapsto P_J$ is continuous, then so is $J\mapsto f_{J}=P_J^N\in C^1(V,\mathbb R^s)$. Thus, there exists a $C^1$-open neighborhood  $\mathcal A\subset \mathscr I_{D,\hat D}$ of the impulse 
$I$ so that  $f_{J}\in \mathcal U$ for every $J\in \mathcal A$, and so it has a unique hyperbolic periodic point $x_1^J \in V$ which is the continuation of $x_1$.
This guarantees that there exists a unique hyperbolic periodic orbit $\gamma_J$ for the impulsive semiflow $\psi_J$, which is contained in the neighborhood 
$\mathcal V$ with period in the interval $(T-N\delta, T+N\delta)$. Moreover, as the first hitting time $\tau(\cdot)$ is continuous on $\bigcup_{i=1}^N V_i$, the periods of the periodic orbits
$\gamma_J$ depend continuously with respect to $J$ . Therefore, for any $\varepsilon>0$ one can diminish $\mathcal A$ if necessary so that the periods
of the periodic orbits $\gamma_J$ belongs to the interval $(T-\varepsilon, T+\varepsilon)$. This proves item (2) and completes the proof of the lemma.
\end{proof}
\color{black}

The following instrumental result will allow us to perturb periodic orbits whose orbit closure intersect the interior of the impulsive region, so that it persists and becomes hyperbolic. 

\begin{lemma}\label{abertura}
Let $\psi_I$ be an impulsive semiflow with impulsive map $I$ and $\gamma$ a periodic orbit for $\psi_I$ so that 
its orbit closure $\overline{\gamma}$ satisfies 
$\overline{\gamma} \cap D\neq\emptyset$ and
$\overline{\gamma} \cap \partial D=\emptyset$.
%\color{bgreen}$\overline{\gamma} \cap  D\neq \emptyset$, 
%\color{black}$\gamma \cap \hat D\neq \emptyset$
%and $\gamma\cap \partial \hat D=\emptyset$.
%$p \in Per(\psi_I)$. 
Then for any $\delta>0$ there exists an impulsive map $J \in \mathscr I_{D,\hat D}$ that is $C^1$-$\delta$-close to $I$ and such that $\gamma$ is an hyperbolic periodic orbit for $\psi_J$.
\end{lemma}

\begin{proof}
Since $\gamma$ is a periodic orbit of $\psi_I$ 
so that $\overline{\gamma} \cap \partial D=\emptyset$ 
and $I(\partial D)=\partial \hat D$ then $\gamma \cap \partial \hat D =\emptyset$. In particular
there exist 
points $x_1, x_2, \dots, x_N \in \gamma \cap \text{interior}(\hat D)$ and pairwise disjoint open neighborhoods $(V_i)_{1\le i \le N}$ of these points in $\hat D$ 
given according to Definition~\ref{defhyp} so that $P(x_i)=x_{i+1}$ for every $1\le i \le N-1$
and $P(x_N)=x_1$. 
Moreover, the $C^1$ Poincar\'e maps $P_I\mid_{V_i}: V_i \to P_I(V_i)\subset \hat D$ are $C^1$-smooth. 
Diminishing $V_i$, if necessary, one guarantees that $P(V_i)\subset V_{i+1}$
for every $1\le i \le N-1$ and the collection 
$$
W_i=\big\{ \varphi_{\tau(x)}(x) \colon x\in V_i\big\}
\qquad (1\le i \le N)
$$
of open subsets of $D$ is pairwise disjoint.
Using ~\eqref{defPo00}, one observes that $x_1\in V_1$
is a fixed point for the map $f: V_1 \to P_I^N(V_1)$ given by
$$
f(x)= P_I^N(x) = P_I(P_I^{N-1}(x)) = I \circ \varphi_{\tau(P_I^{N-1}(x))}(P_I^{N-1}(x)),
\qquad x\in V_1.
$$
Consider the point $\tilde x_N=\varphi_{\tau(P_I^{N-1}(x_1))}(P_I^{N-1}(x_1))\in D$.
If $x_1$ is not a hyperbolic fixed point for $f$ then, Franks' lemma applied to a single point 
(see \cite{Franks} for details) %(Theorem~\ref{Franks}) 
implies that for any $\delta>0$ there exists $\xi>0$ so that for any matrix $A$ 
satisfying $\|A-DI(\tilde x_N)\|<\xi$ there exists a diffeomorphism $J: D \to \hat D$,
$C^1$-$\delta$ close to $I$ so that $J(\tilde x_N)=x_1$, that $J(x)=I(x)$ for every
$x\in \bigcup_{1\le i \le N-1} W_i$ and $DJ(\tilde x_N)=A$. 
%where to the set $S=\{x_1, \cdots, x_N\}$  and linear isomorphisms $L_i=DI(p_i)$ for all $1\le i\le n-1$ and $L_n=A$ is a matrix close to $DI({p_n})$ (to be chosen) one obtains a $C^1$-map $J$ which is $C^{1}$ close to $I$ and such that  $J$ coincides with $I$ in $S$ and $DJ({p_1})= A$. 
Therefore, $x_1$ is a fixed point for the local diffeomorphism
$$
g= J \circ \varphi_{\tau(P_I^{N-1}(\cdot))}(P_I^{N-1}(\cdot))
$$
and 
$Dg(x_1)= A \cdot DI(\tilde x_N)^{-1} \cdot Df(x_1)$ which can be chosen hyperbolic for a suitable choice of the matrix $A$ because the map $A\mapsto A\cdot DI(\tilde x_N)^{-1} \cdot Df(x_1)$ is a submersion in $\text{GL}(d,\mathbb R)$. We conclude that $\gamma$ is a hyperbolic periodic orbit for the impulsive semiflow $\psi_J$, as desired.
\end{proof}

We can now derive the main result of this subsection.

\begin{proposition}
    \label{Kupka-Smale}
There exists 
%\mathscr I^{\mathcal T}_{D,\hat D}$
%\mathcal{I}:= \{I \in C^1(D, \hat{D}) : I(D) \pitchfork X\}$ 
a Baire generic $\mathscr{R}_{per}\subset 
\mathscr I_{D,\hat D}$ so that, for every $I\in \mathscr{R}_{per}$, all $\psi_I$-periodic orbits $\gamma$
%that intersects $\hat D$ and 
such that
$\overline{\gamma}\cap (D \setminus \partial D) \neq \emptyset$ 
are hyperbolic.
% the set of periodic points of the impulsive flow coincide with its set of hyperbolic periodic points:
% $$\Per(\psi_I)= Per_h(\psi_I).$$
%\overline{Per_h(\psi_I) \cap D} = {\Omega(\psi_I) \cap D} 
\end{proposition}

\begin{proof}
Given $\varepsilon>0$ let $\hat D_\varepsilon\subset \hat D$ denote the closed cross-section formed by points $x\in \hat D$ so that $d(x,\partial \hat D)\ge \varepsilon$.
For each $n\ge 1$, denote by $\mathcal{O}_{n,\varepsilon}\subset \mathscr I_{D,\hat D}$ the set of impulses $I\in \mathscr I_{D,\hat D}$ so that $\psi_I$ has finitely many periodic orbits intersecting $\hat D_\varepsilon$, all of them of period smaller or equal than $n$ and hyperbolic. %$\text{interior}(\hat D)=\emptyset$ 
Notice that every $I\in \mathcal{O}_{n,\vep}$ has finitely many such hyperbolic periodic orbits and that, by Lemma~\ref{cont.hip}, the set $\mathcal{O}_{n,\varepsilon}$ 
is a $C^1$-open subset of $\mathscr I_{D,\hat D}$.

By the assumptions that $\varphi$ is a flow and $D\cap \hat D=\emptyset$, there exist $\tau_0,\delta_0>0$  
so that $\tau(x)\ge \tau_0$ for every $x\in \hat D$ and $\hat D\cap \{\varphi_{-t}(\hat D) \colon t\in (0,\delta_0]\}=\emptyset$. 
The choice of $\tau_0$ guarantees that 
every periodic orbit 
$\gamma$ of period smaller or equal than $n$ which intersects $\hat D_\varepsilon$ has at most 
$\lfloor \frac{n}{\tau_0} \rfloor $
intersections with $\hat D_\varepsilon$. 
Together with the fact that $\hat D$ a cross-section to the flow, the previous choices 
%Both choices of $\tau_0,\delta_0$
ensure that 
 $
\hat D \cap \bigcup_{t=0}^n \varphi_{-t}(\partial D)
$ 
consists of the union of finitely many codimension-one connected submanifolds of $\hat D$. 
%the bound for the cardinality could be approx n/\vep_0 where \vep_0 %is given by local flowchart close to hat D that %does not intersect D
%(whose number is bounded by $\lfloor \frac{n}{\tau_0}\rfloor$ or \tau_0/\vep_0) depending on the geometry then we conclude that $L_j$ is finite.
In this way, for each $j\ge 1$ one obtains a decomposition
%$$
%\hat D_\varepsilon = \bigcup_{\ell \in L_j} \hat %D_{\varepsilon,\ell} \; \cup \; E_j
%$$
$$
\hat D_\varepsilon = \bigcup_{B \in \widehat{\mathcal D}_j} B \; \cup \; E_j
$$
where $E_j=\{x\in \hat D_\varepsilon \colon \tau_j(x)=+\infty\}$ is the set of points that return less than $j$ iterates to the cross section $D$ under the orbit of $\psi_I$
and $\widehat{\mathcal D}_j$ is a finite collection of subsets of $\hat D_\vep$
so that $P_I^j: B \to \hat D_{\varepsilon}$ is a well defined $C^1$ diffeomorphism onto its image 
(extendable to the closure of $B$). By construction, the boundary of the sets $B$ is formed by points of $\hat D$ whose forward orbit by $\varphi$ intersects $\partial D$.

Consider the set
$$
\mathscr{T}_{n,j,\varepsilon}= \bigcap_{B \in \widehat{\mathcal D}_j} \,\big\{ I\in \mathscr I_{D,\hat D} \colon (id, P_I^j)\mid_{\overline{B}}  \; \pitchfork \Delta_\varepsilon \big\}
$$
where $\Delta_\varepsilon$ is the diagonal in $\hat D_\varepsilon\times \hat D_\varepsilon$.
By Thom's transversality theorem (cf. \cite{Thom}) we have that each set 
$\,\big\{ I\in \mathscr I_{D,\hat D} \colon (id, P_I^j)\mid_{\overline{B}} \; \pitchfork \Delta_\varepsilon \big\}$ is $C^1$-open and dense, hence
$$
%\bigcap_{\varepsilon \in \mathbb Q_+} 
\bigcap_{j=0}^{\lfloor \frac{n}{\tau_0} \rfloor } \; \bigcap_{B \in \widehat{\mathcal D}_j} \,\big\{ I\in \mathscr I_{D,\hat D} \colon (id, P_I^j)\mid_{\overline{B}}\;  \pitchfork \Delta_\varepsilon \big\}
$$
is a $C^1$ Baire generic subset of $\mathscr I_{D,\hat D}$, formed by impulses $I$ which have only hyperbolic periodic points intersecting 
$\hat D_\varepsilon$. This proves that $\mathcal{O}_{n,\varepsilon}$ is $C^1$-dense.
%$\mathcal{R}_n:= \{ I\in \mathcal{I}; Per_{\leq n}(\psi_I)  \ \mbox{are all hyperbolic}\}$ is open.
%Now, consider the map $\mathscr O\ni I \mapsto $
%Let $x\in Per(\psi_I)\cap \hat{D}$. If $x$ is periodic then by Lemma~\ref{cont.hip}, $x$ admits a hyperbolic continuation. 
%If  $x$  is not hyperbolic  by Lemma~\ref{abertura} there exists $\varepsilon >0$ there exists $J\in \mathcal{I}$, $C^1$-close to $I$, such that $x \in Per_h(\psi_J)$. 
Altogether we conclude that
$\mathcal{R}=\bigcap_{\varepsilon \in \mathbb Q_+} \bigcap_{n \ge 1} \mathcal{O}_{n,\varepsilon}$ is a Baire residual set  that satisfies the conclusion of the proposition. 
\end{proof}

\section{Abundance of periodic orbits}\label{sec:GDTDDD}

This section is devoted to the proof of Theorem~\ref{thmAA}. In Subsection~\ref{sec:perturbC1} we prove 
a connecting lemma for impulsive semiflows, which may be of independent interest assuming
the existence of perturbation boxes for impulsive semiflows. The existence of such perturbation boxes is guaranteed at the Appendix A. In Subsection~\ref{sec:proofthmAA} 
we will use the connecting lemma to conclude the proof of Theorem~\ref{thmAA}.

%%%%%%%%%%%%%%%
\subsection{A connecting lemma}\label{sec:perturbC1}

%In this subsection we recall two $C^1$-perturbation lemmas for discrete-time dynamical systems which will be extremely useful. The first one, known as Franks' lemma, will be used to perform arbitrary small perturbations of a diffeomorphism to preserve a specific periodic orbit and to make it hyperbolic after 
%perturbation. 
%
%\begin{theorem}[Franks Lemma]\label{Franks}
%Let $N$ be a \color{red} compact \color{black} Riemmanian manifold (possibly with boundary),
%let $f \in \text{Diff}^{\,1}(N)$ and fix $S = \{p_1,...,p_m\} \subset \text{interior}(N)$.  
%%
%For any  $C^1$ open neighborhood $\mathcal{U} \subset \text{Diff}^{\,1}(N)$ of $f$ 
%there exists $\delta > 0$ such that if $ \{L_i : T_{p_i}N \to T_{p_{i+1}}N\}_{i=1}^m$ are linear maps %belonging to $GL(d,\mathbb R)$  
%satisfying $\|L_i-  Df(p_i)\| \leq  \delta$ for $i = 1, \cdots, m$ then there exists $g \in  \mathcal{U}$ 
%such that $g(p_i) = f(p_i) $  and $Dg(p_i) = L_i$.
%\end{theorem}

%Suppose $D$ is a submanifold of codimension one and has no boundary. 
%First we recall some necessary notions.
Given $\delta>0$, we say that a sequence $(x_k,t_k)_{k=0}^n$  in $M\times \mathbb R_+$ is a $\delta$-\emph{pseudo orbit} for the impulsive semiflow $\psi_I$
if 
$$
d(\gamma_{x_k}(t_k), x_{k+1})<\delta, \quad\text{for every} \; k=0 \dots n-1.
$$
We say that $y$ is a \emph{chain iterate} of $x$ (and write $x \dashv y$) if for any $\delta>0$ there exists a $\delta$-pseudo orbit
$(x_k,t_k)_{k=0}^n$ such that $x_0=x$ and $x_n=y$.

\begin{theorem}\label{thm:closing}
Let $\varphi$ be a $C^1$-flow generated by  $X\in \mathfrak{X}^1(M)$, let $D,\hat D$ be smooth submanifolds of 
codimension one transversal to $X$ and $I \in  \mathscr I^{\mathcal T}_{D,\hat D}$ be a $C^1$-impulse so that 
all $\psi_I$-periodic orbits 
%of the impulsive semiflow $\psi_I$ that do not intersect the boundary of $\hat D$
whose orbit closure does not intersect $\partial D$ are hyperbolic. 
If $x \dashv y$
then for any $\vep>0$ there exists an $\vep$-$C^1$-perturbation $J$ of $I$ such that $y=\gamma_{J,x}(t)$, for some $t\ge 0$. 
\end{theorem}

Some comments are in order. In the context of diffeomorphisms such connecting lemmas were obtained by Arnaud ~\cite[Th\'eor\`eme~22]{Arn} and Bonatti and Crovisier \cite[Th\'eor\`eme~2.1]{BC04}, as refinements of the $C^1$-closing lemma by Pugh \cite{Pu} and the $C^1$-connecting lemma by Hayashi \cite{Hay}. 
Let us discuss briefly some key differences between the context of diffeomorphisms and the one of impulsive semiflows.
On the one hand the natural discrete time dynamical system
\footnote{
By definition of the impulsive semiflow, one cannot consider a Poincar\'e map whose domain is the set $D$ of the form $P\circ I$, 
as the orbits of these points are given (locally) by the original flow.} 
is the piecewise $C^1$ Poincar\'e map
\begin{equation*}\label{defPo00712}
\begin{array}{cccc}
P_I : & \widehat{I(D)}  & \to & \hat D \\
	& x & \mapsto & I \circ \varphi_{\tau(x)}(x)
\end{array}
\end{equation*}
whose domain is the subset $\widehat{I(D)} \subset \hat D$, which does not intersect the set $D$, where the perturbations of impulses are supported. 
On the other hand, as the $C^1$-perturbations performed in the proof of the connecting lemma for a diffeomorphism $f$ are supported on a finite collection of sets of the form $\{f^j(U): 1\le j \le N\}$, this means that the perturbation boxes inside $D$ should take into account iterates of the form $(I \circ \varphi_{\tau(\cdot)})^n$, hence these depend not only on points of $D$  
as on their images by the impulsive semiflow $\psi_I$.
In this way, Theorem~\ref{thm:closing} cannot be obtained as a consequence of \cite{Arn,BC04}.

%\begin{corollary}\label{thm:pugh}
%Let $\varphi$ be a $C^1$-flow generated by  $X\in \mathfrak{X}^1(M)$, let $D,\hat D$ be smooth submanifolds of 
%codimension one transversal to $X$ and $I \in \mathscr I_{D,\hat D}$ be a $C^1$-impulse so that all periodic orbits
%of the impulsive semiflow $\psi_I$ that do not intersect the boundary of $\hat D$ are hyperbolic. 
%If {$p \in \Omega(\psi_I) \cap \text{interior}(\hat{D})$} then there exists an impulse $J\in \mbox{Diff}^{\;1}(D, \hat D)$ that is $C^1$-$\varepsilon$-close  of $I$ and such that $p \in \mbox{Per}(\psi_J)$.
%\end{corollary}

\smallskip
The previous connecting lemma
%, i.e. Theorem~\ref{thm:closing} one can derive 
yields the following consequence:

\begin{corollary}\label{thm:pugh}
Let $\varphi$ be a $C^1$-flow generated by  $X\in \mathfrak{X}^1(M)$, let $D,\hat D$ be smooth submanifolds of 
codimension one transversal to $X$ and $I \in \mathscr I^{\mathcal T}_{D,\hat D}$ be a $C^1$-impulse so that all periodic orbits
of $\psi_I$ 
whose orbit closure does not intersect $\partial D$ are hyperbolic. 
%that do not intersect the boundary of $\hat D$ are hyperbolic. 
Assume that {$p \in \Omega(\psi_I) \cap \text{interior}(D)$} and let $q=I(p)\in \text{interior}(\hat D)$.
For any $\vep>0$ there exists $J\in \mathscr I^{\mathcal T}_{D,\hat D}$ that is $C^1$-$\varepsilon$-close of $I$ such that $q \in \mbox{Per}(\psi_J)$ and its orbit closure contains $p$.
%If {$p \in \Omega(\psi_I) \cap \text{interior}(D)%$} then there exists 
%$J\in \mathscr I^{\mathcal T}_{D,\hat D}$ that is $C^1$-$\varepsilon$-close of $I$ and there exists $\tilde{p}$ close to $p$ such that $\tilde{p} \in \mbox{Per}(\psi_J)$ and the closure of its orbit intersects the interior of $D$.
\end{corollary}

\begin{proof}
 Fix $p \in \Omega(\psi_I) \cap \text{interior}(D)$ and $q=I(p)$. Since $I$ is a diffeomorphism and $I(\partial D)=\partial \hat D$ it is not hard to check that $q\in \Omega(\psi_I) \cap \text{interior}(D)$. In particular, given $\delta>0$ there exists $y\in \hat D$ such that $d_{\hat D}(y,q)<\delta$ and a large $n\ge 1$ so that $d_{\hat D}(P_I^n(y),q) < \delta$. 
As $\delta>0$ was chosen arbitrary this implies that $q \dashv q$.
Using Theorem~\ref{thm:closing} we conclude that there exists an impulse $J\in \mathscr{I}^{\mathcal T}_{D, \hat D}$ that is $C^1$-$\varepsilon$-close  to $I$ such that $q \in \mbox{Per}(\psi_J)$, as desired. 
%Given  $p \in \Omega(\psi_I) \cap \text{interior}(\hat{D})$ take $\delta_0>0$ so that $B_{\hat D}(p,\delta_0) \subset \hat D$,
%where $B_{\hat D}(p,\delta_0)=\{y\in M \colon d_{\hat D}(y,p)<\delta_0\}$. Fix an arbitrary $0<\delta<\delta_0$. By definition of the non-wandering set, there exists $y\in B_{\hat D}(p,\delta)$ and a large $n\ge 1$ so that $P_I^n(y) \in B_{\hat D}(p,\delta)$, hence
%$
%d_{\hat D}(y,P_I^n(y))<  2\delta.
%$
%As $\delta>0$ was chosen arbitrary this implies that $p \dashv p$.
%Using Theorem~\ref{thm:closing} we conclude that there exists an impulse $J\in \mbox{Diff}^{\;1}(D, \hat D)$ that is $C^1$-$\varepsilon$-close  of $I$ and such that $p \in \mbox{Per}(\psi_J)$. This proves the corollary.
\end{proof}

%\color{red}
%
%
%As $P_I^n(\cdot)=I \circ \varphi_{\tau(\cdot)} \circ P_I^{n-1}(\cdot)$, the point $I^{-1}(P_I^n(y))$ belongs to the closure of the $\psi_I$-orbit of the point $I^{-1}(y) \in D$. Moreover, 
%$$
%d_{ D}(I^{-1}(y),I^{-1}(P_I^n(y)))<  2 \delta \| I^{-1} \|_{C^1}. 
%$$
%This shows that even though the points defined by
%$$
%y_n^D=I^{-1}(P_I^n(y)) 	\quad\text{and}\quad  y^D_{j-1}= I^{-1}(y^D_{j})
%$$
%where $\tau^-: D \to \mathbb R $ is defined by
%$$
%\tau^-(y)= \inf\{t>0 \colon \varphi_{-t}(y) \in \hat D\} \qquad y\in D
%$$
%do not belong to the orbit of $y_0$ by the impulsive semiflow $\psi_I$, these belong to the orbit closure of $y_0$
%and 
%
%\vspace{2cm}
%
%$I^{-1}$
%and 
%$$
%	\quad\text{and}\quad
%d(I^{-1}(y), I^{-1}\circ P_I^n(y)) 
%	=
%	\le K d_D(I^{-1}(y), I^{-1}\circ P_I^n(y)) \delta
%$$
%(here $K>0$ is a uniform constant depending on the curvature of $D$).
%
%\

\color{black}

\medskip
The proof of Theorem~\ref{thm:closing}, which
builds over ideas from ~\cite{Arn,BC04} adapted to the context of impulsive semiflows, will now occupy the remainder of this subsection.

\begin{proof}[Proof of Theorem~\ref{thm:closing}] 
There are two main ingredients in the proof of the theorem, and these involve the construction of perturbation boxes of uniform length and a selection lemma to extract redundancies on certain pseudo-orbits. Using the latter, the theorem is deduced by perturbing the initial impulse by $C^1$-small perturbations performed on disjoint domains inside $D$.  
As the proof  follows closely the one of \cite[Theorem~1.2]{BC04} we shall give a brief sketch of the argument, stressing the necessary modifications: (i) perturbation boxes are defined on the domain $D$ of the impulsive maps,  (ii) the perturbation length depends both on the initial flow and the impulsive map, and (iii) $C^1$-perturbations of the impulse $I$ are of the form $J=h\circ I$, where $h: \hat D \to \hat D$ is a $C^1$ perturbation of the identity (cf. Proposition~\ref{thm:cubes} and Appendix A). These
%, together with an abstract selection lemma for pseudo-orbits (cf. Lemma~\ref{le:selectingl} below) 
will be the main tools to prove the connecting lemma.

\smallskip
The first step consists of finding an integer $N\geq 1$ and perturbation boxes $B$ (of length $N$) so that pseudo-orbits of length $N$ preserving the perturbation box can be connected by $C^1$-small perturbations (this step corresponds to ~\cite[Th\'eor\`eme~2.1]{BC04} in case of diffeomorphisms). Let us define the perturbation boxes 
inside the impulsive region $D$, in the spirit of \cite[page~51]{BC04}.
Set $d=\dim M$. 
For each small open set $B\subset D$ and local chart $\Theta: B \to \mathbb R^{d-1}$,  we say that $B$ is a \emph{tiled cube} if
$\Theta(B)\subset \mathbb R^{d-1}$ coincides (up to translation or homothety) with the cube $]-3,3[^{d-1}$ tiled  by a finite number of 
smaller cubes $C$, 
where either $C=[-1,1]^{d-1}$ or it is of the form
$$
C= \prod_{j=1}^{d-1} \Big[  \frac{k(j)}{2^i}, \frac{k(j)+1}{2^i}  \Big]
$$
for some $i\ge 0$ and integers $k(j) \in [-2^i\alpha_i, 2^i\alpha_i-1]$ 
%for every $1\le j \le d-1$ 
chosen in such a way that 
$k(j_0)\in \{-2^i\alpha_i, 2^i\alpha_i-1\}$ 
for some $1\le j_0\le d-1$ 
(here $\alpha_i=1+\sum_{j=0}^i 2^{-j}$). 
 Each tile $\Theta^{-1}(C)$ will be referred to as a cube of $\Theta$ whenever it will become useful to think it
 is mapped (up to translation and homothety) on the standard cube $[-1,1]^{d-1}$ (see Figure~1 below).  By some abuse of notation, when no confusion is possible we shall omit the dependence on the charts and simply denote the cubes in the impulsive region 
 $D$ by $C$ instead of $\Theta^{-1}(C)$ for some chart $\Theta$ and for the same reason we shall denote by $\alpha C$ the set $\Theta^{-1}(\alpha C)$
 for each $\alpha>0$.

\begin{figure}[htb]\label{fig1}
\begin{center}
  \includegraphics[width=10cm,height=5cm]{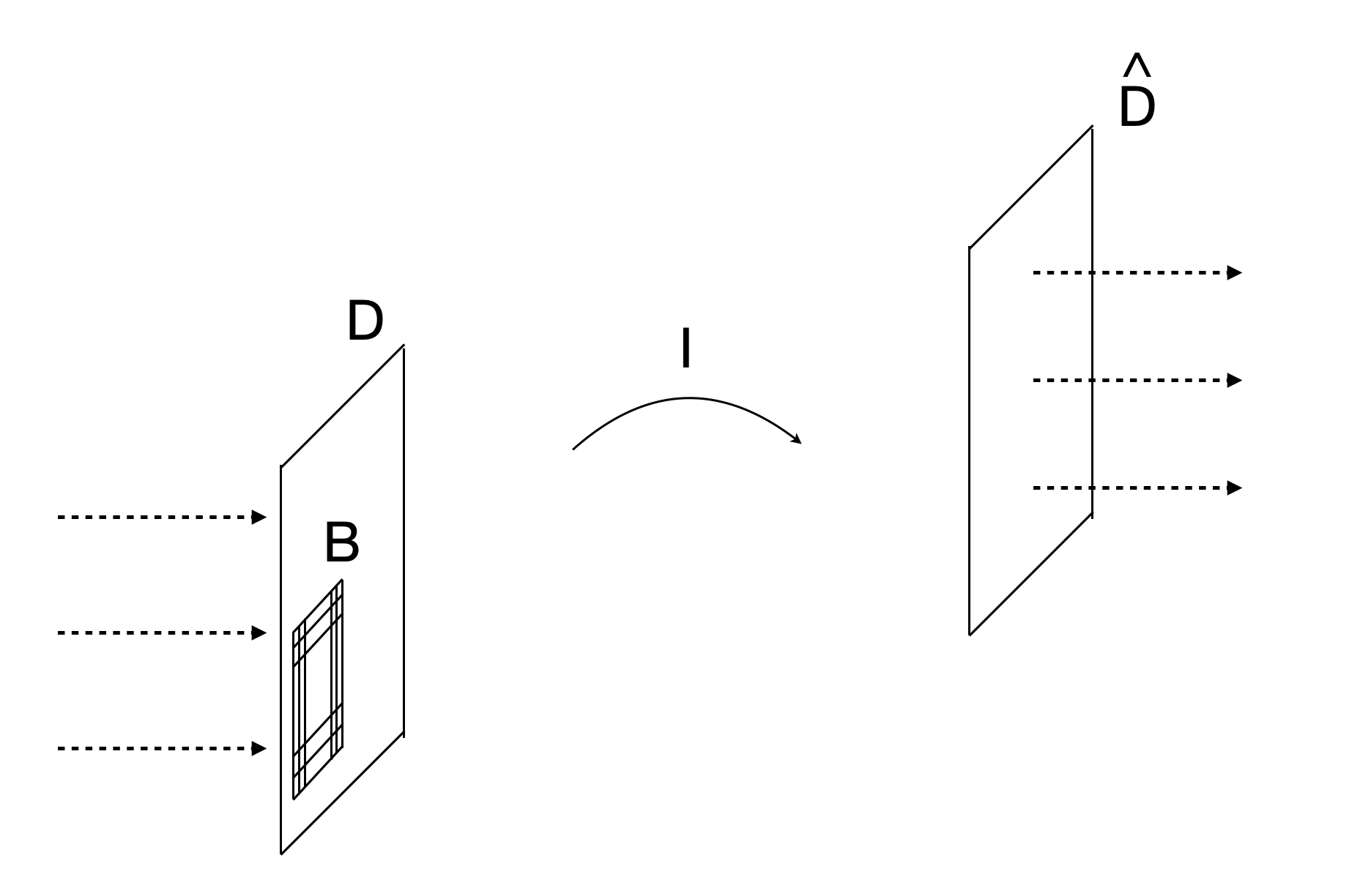}
  \caption{Tiled cube $B$ inside the impulsive region $D$}
\label{figure}
\end{center}
\end{figure}

In what follows, for each impulse $I\in \mathscr{I}_{D,\hat D}$ we shall consider $f_I\colon D\to D$ given by $f_I=I^{-1}\circ P_I\circ I$, which is conjugate to the Poincar\'e map $P_I$ but whose domain is the impulsive region $D$.
Notice that $f^k_I=I^{-1}\circ P^k_I\circ I$ for every $k\ge 1$.

\begin{definition}(Perturbation box of order $N$)\label{def:perturb-cube}
Let $\mathcal U$ be a $C^1$-open neighborhood of an impulse $I \in \mathscr I_{D,\hat D}$ . Given an integer $N\ge 1$ and 
constants $\vep,\eta \in (0,1)$, we say that $C$ is a $(I,\mathcal U, \Theta, N, \vep, \eta)$-\emph{perturbation box of order $N$} if
 the  sets $(f_I^j(C))_{0\le j \le N}$ form a pairwise disjoint collection in the interior of $D$ and for any pair of points $(p,q)$ in $C$ there exists $J\in \mathcal U$ so that:
\begin{enumerate}
\item %$J$ coincides with $I$ outside of the set $\bigcup_{\ell=0}^{N-1}  (I^{-1}\circ P_I^\ell\circ I) ((1+2\vep)C)$;
	$J$ coincides with $I$ outside of the set $\bigcup_{\ell=0}^{N-1}  f_I^\ell ((1+2\vep)C)$;
\item $f_J^N  (p)=f_I^N (q)$; 
\item $f_J^\ell (p) \in f_I^\ell ((1+\vep)C)$ for each $0\le \ell < N$;
\item for each $0\le \ell < N$ the impulsive maps $J$ and $I$ coincide on the region
$f_I^\ell ((1+2\vep)C)$ except possibly on a small ball centered at the point 
$f_J^\ell  (p)$ and of radius $\eta$ times the distance between  the set
$f_I^\ell ((1+\vep)C)$ and the complement of $f_I^\ell ((1+2\vep)C)$.
\end{enumerate} 

The support of a perturbation box $C$ of order $N$ is the set 
$
\text{supp}(C)= \bigcup_{j=0}^{N} f_I^j(C).
$
\
\end{definition}

One of the key ingredients in the proof of the theorem is the following counterpart of Th\'eor\`eme A.2 in \cite{BC04} on the existence of perturbative cubes (we refer the reader to Appendix A for the proof).

\begin{proposition}(Uniform perturbation boxes)\label{thm:cubes}
Let $M$ be a compact Riemannian manifold, $\varphi$ be a $C^1$-flow generated by  $X\in \mathfrak{X}^1(M)$, 
$D,\hat D$ be smooth, compact submanifolds of  codimension one transversal to $X$ and let $I \in \mathscr I_{D,\hat D}$. 
%be a $C^1$-impulse. 
Fix $\vep, \eta \in (0,1)$.  For every $C^1$-open neighborhood $\mathcal U\subset \mathscr I_{D,\hat D}$ of $I$ there exists $N\ge 1$
 so that every point   $x\in D$ admits an open neighborhood $V\subset D$ and a chart $\Theta: V\to R^{d-1}$ satisfying: every cube $C$ of $V$ such that 
 the first $N$ iterates by  $f_I$ of the cube $(1+ 2\vep)C$ are pairwise disjoint and lie 
 in the interior of $D$ is a $(I,\mathcal U, \Theta, N, \vep, \eta)$-perturbation box.
\end{proposition}

%A sketch of the proof of Proposition~\ref{thm:cubes} will be given at Appendix~A. ; Not really

\medskip

Perturbation boxes will allow to shadow certain classes of pseudo-orbits, up to perturbation of the impulse. Let us be more precise.

Recall that a collection of points $x_0, x_1, \cdots,x_n $ in $D$ is an \emph{$\vep$-pseudo-orbit} for the map $f_I$
 if $d(f_I(x_i),x_{i+1})<\vep$ for every $0\le i \le n-1$.
Assume that %$\mathcal U\subset \mathscr I_{D,\hat D}$  is a $C^1$-open neighborhood of $I$, 
$B\subset D$ is a perturbation box of order $N$, that $x, y \in D\setminus \text{supp}(B)$ and that  $x_0=x, x_1, \cdots,x_n=y $
is an  $\vep$-pseudo-orbit.  We say that the pseudo-orbit $\{x_n\}$ \emph{preserves the grid of the perturbation box $B$} if the intersection of the pseudo-orbit with the support of $B$ is an union of points $x_i, x_{i+1}, \cdots, x_ {i+k}, \cdots x_ {i+N}$ so that $x_i \in B$ and $x_{i+k}= f_{I}^{k} (y_i)$ for each $k\in \{1, \cdots N \}$,  where $y_i$  is a point of $B$ that belongs to the same tile as $x_i$.
%Moreover  the pseudo-orbit  is said to preserve the perturbation box  without jumps inside $supp(B)$ if the intersections of the pseudo-orbit with the support of $B$ is an union of points $x_i, x_{i+1}, \cdots, x_ {i+k}, \cdots x_ {i+N}$ of the form $x_i \in B$, $x_{i+k}= f^k(x_i)$, $k\in \{1, \cdots N \}$.
In the special case that $y_i=x_i$ we say that the
pseudo-orbit \emph{has no jumps inside $\text{supp}(B)$}.
%
%\begin{definition}
%Let  $\mathcal U\subset \mathscr I_{D,\hat D}$  be a $C^1$-open neighborhood of $I$.
%Let $B$ be  a perturbation box of order $N$. Let $x, y \notin supp(B)$. Given  an  $\vep$-pseudo-orbit  $x_0=x, x_1, \cdots,x_n=y $ we say that  the pseudo-orbit $\{x_n\}$ preserves the grid of the perturbation box $B$ if the intersection of the pseudo-orbit with the support of $B$ is an union of points $x_i, x_{i+1}, \cdots, x_ {i+k}, \cdots x_ {i+N}$ of the form $x_i \in B$, $x_{i+k}= (I^{-1}\circ P_{I}^{k} \circ I) (y_i)$, $k\in \{1, \cdots N \}$,  or $y_i$  is a point of $B$ that belongs to the same tile as $x_i$.
%
%Moreover  the pseudo-orbit  is said to preserve the perturbation box  without jumps inside $supp(B)$ if the intersections of the pseudo-orbit with the support of $B$ is an union of points $x_i, x_{i+1}, \cdots, x_ {i+k}, \cdots x_ {i+N}$ of the form $x_i \in B$, $x_{i+k}= f^k(x_i)$, $k\in \{1, \cdots N \}$.
%\end{definition}
%
%Para concluir a prova me parece que o essencial eh o Lemma 2.3 e Prop. 4.8 do Bonatti-Crovisier.

\begin{lemma}\label{lema2.3}
Let $\mathcal{B}$ be a collection of perturbation boxes of order $N$ in $D$ whose supports are pairwise disjoint. Assume that $x,y\in D$ are points that do not belong to the support of these perturbation boxes and $x_0=x, \ldots, x_n=y$ is an $\vep$-pseudo-orbit that preserves the grid of the perturbation box and has no jumps  outside of the support of the perturbation boxes in $\mathcal{B}$. Given $B\in \mathcal{B}$  %of order $N$, 
there exist an impulsive map $J \in\mathcal{U}$ that coincides with $I$ outside of $\text{supp}(B)$ and an $\vep$-pseudo-orbit $z_0=x, z_1, \cdots,z_m=y $ associated to the perturbed map $f_J$ which preserves the grid without jumps outside the cubes of $\mathcal{B}\setminus \{B\}$.
\end{lemma}

\begin{proof}
This result is the counterpart of Lemme~2.3 in \cite{BC04} to our current setting and, since it does not involve any perturbation argument, it follows as a direct consequence of
the definition of perturbation box (recall Definition~\ref{def:perturb-cube}) and the abstract selection of pseudo-orbits in 
\cite[Remarque~2.2]{BC04} applied to the discrete dynamical system $f_I$ and impulsive region $D$ (see also Figure~2 in \cite{BC04}).
\end{proof}

%\begin{proposition}\label{Prop4.8}
%\marginpar{\tiny eh este o principal em falta?}
%There exists $\vep_1$ such that for all points $x,y$ outside the neighborhoods $V(\gamma)$ and outside the supports of the perturbation boxes $B\in \mathcal{B}$ we have the following property:  
%if $x \dashv y$ then there exists a pseudo-orbit preserving the grid  of all boxes of $\mathcal{B}$ and without jumps outside the boxes of $\mathcal{B}$ joining $x$ to $y$. 
%\end{proposition}
%
%\begin{proof}
%Vamos precisar adaptar a prova da Prop 4.8  do Bonatti-Crovisier.
%\end{proof}
%\

%The idea is to partition the domain of the impulse $D$ in order to produce suitable perturbation boxes for the impulse map $I$.
%
%\begin{lemma}(Selecting lemma)\label{le:selectingl}
%\end{lemma}
%
%\begin{proof}
%{\color{red} \tiny \sc missing }
%\end{proof}

\medskip
We are now in a position to complete the %sketch of 
the proof of Theorem~\ref{thm:closing}.
Assume that $\mathscr{I}^{\mathcal T}_{D,\hat D}\neq\emptyset$, otherwise there is nothing to prove.
Let  $\varepsilon , \eta \in (0,1)$ and consider 
a $C^1$-open neighborhood $\mathcal{U}$
%\subset \mathscr I^{\mathcal T}_{D,\hat D}$ 
of an impulse $\mathcal{I} \in \mathscr I^{\mathcal T}_{D,\hat D}$. 
Diminishing   $\mathcal{U}$, if necessary, we can assume that the following property: 
 given $V_1, \cdots, V_r$ disjoint open subsets of $D$  and $J_1, \ldots, J_r \in \mathcal{U}$  such that for all $1\le i\le r$, $J_i$ coincides with $I$ outside $V_i$ one has that if $J$ is such that $J \in \mathscr I_{D,\hat D}$,
 coincides with $J_i$ in $V_i$ and with $I$ outside $\displaystyle\cup_{i=1}^rV_i $  then $J$ also belongs to $\mathcal{U}$. 
 This will be useful as the desired perturbation will be obtained
by a finite number of $C^1$-small perturbations with disjoint supports.

\smallskip
Let $N \ge 1$ be the order of the perturbation boxes provided by Proposition~\ref{thm:cubes} and take $N_0:= 10 \, (d-1)\, N $.
Take $x \dashv y$ and assume that $y$ does not belong to the orbit of $x$ (otherwise there is nothing to prove). 

\smallskip
One can suppose that $x, y \notin Per_{N_0}(f_I)$. 
Indeed, %aqui usaria continuidade uniforme, o que não eh %necessariamente valido, mas podemos usar que a derivada é finita
%numa vizinhança de x e teremos essa continuidade
if $x \in Per_{N_0}(f_I)$ and $\vep$ is small then 
any $\varepsilon$-pseudo-orbit joining $x$ to $y$ will have a finite number of points that lie in a small open neighborhood of $x$. Then, any accumulation point  $x'$ of the sequence of points belonging to such $\varepsilon$-pseudo-orbits 
%joining $x$ to $y$ 
and that belong to a neighborhood of $x$
(as $\varepsilon$ goes to zero) 
is not periodic and lies in the local unstable manifold of $x$ 
%arbitrarily close to $x$ 
and satisfies $x'\dashv y$.
In this case it is enough to show the result for $x'$ and $ y$. Indeed, 
 if $J'$ satisfies the result for $x' \dashv y$  then conjugating $J'$ by a $C^1$-small perturbation of the identity whose support does not contains  the point $y$ and that sends  $x'$ to $x$ maintaining the positive orbit of $x$ we have a map $J$ that satisfies the result for $x \dashv y$. 
The case $y \in Per_{N_0}(f_I)$ is analogous.

\smallskip
 Let $\mathcal{B}_0$ be a
 finite family of perturbation boxes of order $N$ as in \cite[Corollaire~4.1]{BC04} (the existence of 
 uniform perturbation boxes is ensured by Proposition~\ref{thm:cubes}).

%\color{red}
% xxxxxx
 %family of perturbation boxes associated to $\mathcal{U}$  given by Proposition~\ref{thm:cubes}  is so  that the union of all the boxes in $\mathcal{B}$  is an open set $U$ disjoint of its $N$ first iterates by $f_I$ and satisfying that every orbit enters in one box of $\mathcal{B}$ in a sufficiently large time. For details on how can one choose the perturbation boxes in the previous way we refer to  Th\'eor\`eme 3.1 in \cite{BC04}. 
%xxxxxx
Let $\mathcal B$ be a finite family of perturbation boxes of order $N$ obtained by enrichment of the family $\mathcal B_0$ by adding perturbation boxes inside open neighborhoods $V(\gamma)$ of periodic orbits $\gamma$ of period smaller or equal than $N_0$,
that is:
$$
\mathcal B=\mathcal B_0 \cup \bigcup_\gamma \mathcal E(\gamma) \cup \mathcal S(\gamma) 
$$
where $\mathcal E(\gamma), \mathcal S(\gamma)$ are perturbation boxes contained in the neighborhood $V(\gamma)$ of $\gamma$  (this follows from Subsection 4.3 in \cite{BC04} applied to $f_I$).
Write $\mathcal B=\{B_1, B_2, \dots, B_\kappa\}$. 

%\color{red}
%The following lemma is an adapted version of Proposition 4.8 in \cite{BC04}.
%

 %xxxxxx
 %Since each box $B\in \mathcal{B}$ is disjoint of its first $N$ iterates  the case of periodic orbits with period smaller or equal than $N_0$ need to be treated separately.  For this, consider $V(\gamma)$  a family of neighborhoods of $\gamma \in Per_{N_0}(f_I) $ that are pairwise disjoints and that are also disjoint of  $\mbox{supp}(\mathcal{B})$. The following lemma is a straightforward version of Proposition 4.8 in \cite{BC04}.
 %xxxxxx

The following lemma is a direct consequence of 
an  abstract result on selection of points in pseudo-orbits  (Proposition 4.8 in \cite{BC04}) applied to points in the local cross-section.

\begin{lemma}\label{pseudo-orbit}
There exists $\vep_1>0$ such that for any periodic orbit $\gamma$ of period smaller than $N_0$ and all points $x,y$ outside the neighborhoods $V(\gamma)$ and outside the supports of the perturbation boxes $B\in \mathcal{B}$ the following holds:  
if there exists an $\vep_1$-pseudo orbit joining  $x$ to $y$ then there exists an $\vep_1$-pseudo-orbit preserving the grid  of all boxes of $\mathcal{B}$ and without jumps outside the boxes of $\mathcal{B}$ joining $x$ to $y$. 
\end{lemma}
%\color{red}
% xxxxxx
%selection result in Proposition 4.8 in \cite{BC04}.
% xxxxxx

Now we claim that one can assume that the points $x \dashv y$ do not belong neither to the boxes $B\in \mathcal{B}_0$  
nor to the neighborhoods $V(\gamma)$. Indeed, if this was not the case, since $y$ does not belong to the trajectory of $x$, we can choose (according to Corollary 4.1 in \cite{BC04}) iterates $f_I^i(x)$ and  $f_I^i(y)$ that do not belong to $B$. 
Notice that $f_I^i(x) \dashv f_I^i(y)$. 
%Hence, to show the conclusion of the theorem, it is enough to prove the result for the points $f_I^i(x)$ and $f_I^i(y)$. In fact, 
In case there exists a $C^1$-small perturbation 
$\widetilde J$ of $I$ such that $f_I^i(y)=\gamma_{\widetilde J,f_I^i(x)}(t)$, for some $t\ge 0$ or, equivalently, there exists $n\ge 1$ so that 
\begin{equation}
    \label{eq:finalle}
f_{\widetilde J}^n(\, f_I^i(x)\,)=\, f_I^i(y)
\end{equation}
 we claim that there exists a $C^1$-small perturbation $J$ of the impulse 
$\widetilde J$ such that $f^n_J(x)=y$.
In fact, as $\dim D=\dim M-1 \ge 2$
and the points $f_{\widetilde J}^{-i}(f_I^i(x))$ and ${f_{\widetilde J}}^{-i}(f_I^i(y))$ are close to $x$ and $y$, respectively, there exists a smooth curve $c_z: [0,1] \to D$ connecting the points $f_{\widetilde J}^{-i}(f_I^i(z))$ 
and $z$, for $z\in \{x,y\}$, and which do not intersect any of the points 
$f_{\widetilde J}^{k}(\, f_{\widetilde J}^{-i}(f_I^i(x))\,)$,  $1\le k \le i+n$.
In particular, there exists a $C^1$-small perturbation $J$ of $\widetilde J$ obtained by isotopy in such a way that: 
\begin{enumerate}
    \item[(i)] $J(x) = \widetilde J(f_{\widetilde J}^{-i}(f_I^i(x)))$ and $J(y) = \widetilde  J(f_{\widetilde J}^{-i}(f_I^i(y)))$
    \item[(ii)] $J$ coincides with $\widetilde J$ at the points $\{f_{\widetilde J}^{k-i}(f_I^i(x))\colon 1\le k \le i+n\}$
    \item[(iii)] $J$ coincides with $\widetilde J$ at the points $\{f_{\widetilde J}^{-k}(f_I^i(y))\colon 1\le k \le i-1\}$
\end{enumerate}
By construction, the impulse $J$ is $C^1$-close to $I$, and 
$$
f_J(x)= \varphi_{\tau(J(x))} (J(x))
    = \varphi_{\tau(\widetilde  J(f_{\widetilde J}^{-i}(f_I^i(x))))}(\widetilde  J(f_{\widetilde J}^{-i}(f_I^i(x)))) 
    = f_{\widetilde J}(\,f_{\widetilde J}^{-i}(f_I^i(x)) \,).
$$
Recursively, using property (ii), one obtains that
$f^j_J(x)= f_{\widetilde J}^{j-i}(f_I^i(x))$ for every 
$1\le j \le n+i$. Taking into account 
equation ~\eqref{eq:finalle} we conclude that
$f^{n+i}_J(x)= f_I^i(y)$. Now, using (iii) and the second
part of item (i),
\begin{align*}
   f^{n}_J(x)   
   & = f^{-i}_{J} (\, f^{n+i}_J(x) \,) 
    = f^{-i}_{J} (\, f_I^i(y) \,) \\
   & = f^{-1}_{J} (\, f^{-i+1}_{J} 
   (\, f_I^i(y) \,)\\
   & = f^{-1}_{J} (\, f^{-i+1}_{\widetilde J} 
   (\, f_I^i(y) \,) = y,
   \end{align*}
thus proving the claim.

Now Lemma~\ref{pseudo-orbit} ensures that there exists a pseudo-orbit $y_0=x, \cdots, y_l=y$  preserving the grid  of all boxes of $\mathcal{B}$ and without jumps. 
This allows us to perform a finite number of perturbations to prove the connecting lemma recursively. 
Indeed, assume that the pseudo-orbit intersects the support $\text{supp}(B_1)$
of the first box $B_1\in \mathcal{B}$
(otherwise proceed to the next box).
Lemma~\ref{lema2.3} (applied to the 
pseudo-orbit $y_0, \cdots, y_l$)  provides an impulsive map  $J_1 \in\mathcal{U}$ which coincides with $I$ outside of $\text{supp}(B_1)$ and an $\vep_1$-pseudo-orbit 
$y_0^{(1)}=x, y_1^{(1)}, \cdots, y_{m_1}^{(1)}=y $ 
for $f_{J_1}$ 
preserving the grid and without jumps.  
If the pseudo-orbit 
$y_0^{(1)}, y_1^{(1)}, \cdots, y_{m_1}^{(1)}$ does not intersect the support
of the second box $B_2\in \mathcal{B}$
then write 
$y_j^{(2)}=y_j^{(1)}$ for each $1\le j \le m_1$. Otherwise, it intersects the support
of the second box $B_2\in \mathcal{B}$ and,
applying Lemma~\ref{lema2.3}, 
we obtain a new pseudo-orbit
$x=y_0^{(2)}, y_1^{(2)}, \cdots, y_{m_2}^{(2)}=y$
preserving the grid and without jumps.  
Applying Lemma~\ref{lema2.3} successively for each box $B_i\in \mathcal{B}$ for $i=2, \ldots, \kappa$ and for the pseudo-orbit 
$x=y_0^{(i-1)}, y_1^{(i-1)}, \cdots, y_{m_{i-1}}^{(i-1)}=y$
of the $(i-1)^{th}$-stage we will have $J_i \in \mathcal{U}$ that coincides with $I$ outside of $\text{supp}(B_i)$. By the choice of $\mathcal{U}$,  the impulsive map $J=J_i$ in $\text{supp}(B_i)$ and $J=I$ outside $\text{supp}(\mathcal{B})$ is so that $J\in \mathcal{U}$ and moreover $y$ belongs to the  orbit of  $x$ under $f_J$ .

   \smallskip

%5) By Proposition~\ref{Prop4.8} there exists a pseudo-orbit  joining $x$ and $y$ that preserves the grid and without jumps. 
%
%\smallskip
%6) Applying the Lemma~\ref{lema2.3} successively  for each box $B\in \mathcal{B}$ we construct the map $J$ satisfying the theorem.
%
This completes the proof of Theorem~\ref{thm:closing}, as desired. 
\color{black}
\end{proof}

%\begin{remark}
%We point out that  in the Closing Lemma  we can consider  a wider class of impulsives maps namely, $I\in \mbox{End}^1(D, \hat{D})$, as long as we consider points in the non-wondering set whose orbits are non-singular for the Poincar\'e map. 
%Moreover, if $I\in \mbox{End}^1(D, \hat{D})$ and the critical set $\mbox{Crit}(I)$ is a submanifold of positive codimension then $\bigcup_{n\geq 0}P_I^{-n}(Crit)(I)$  is a meager set.  
%\end{remark}

\subsection{Proof of Theorem~\ref{thmAA}}
\label{sec:proofthmAA}

Let $D^*$ be the set of  compact  subsets
of $D$  
endowed with the Hausdorff topology, 
%Let $\KS \subset \mathscr I_{D,\hat D}$ be the $C^1$-Baire residual subset given by Proposition~\ref{Kupka-Smale}.
 consider the map 
  $$
	\begin{array}{cccc}
	\Phi \colon \, & 
\mathscr I^{\mathcal T}_{D,\hat D}
  & \rightarrow &  D^\star \\
	& I & \mapsto  & \overline{Per_h(\psi_I)} \cap  D
	\end{array}
	$$
and let $\mathscr{R}_{per}$ be the Baire residual set  given by Proposition~\ref{Kupka-Smale}.
For each $I\in \mathscr{R}_{per}$, 
all periodic orbits 
whose orbit closure intersects the interior of $D$ are all hyperbolic, hence persistent 
(recall Lemma~\ref{cont.hip})
% whose closure intersects $D$ are all hyperbolic 
and  the map $\Phi \mid_{\mathscr{R}_{per}}$ is lower semicontinuous. 
Hence, the continuity points of $\Phi|_{\mathscr{R}_{per}}$ form a residual subset  $\tilde{\mathscr{R}} \subset \mathscr{R}_{per}$. 
We claim that  $\overline{Per_h(\psi_I)} \cap  D = {\Omega(\psi_I) \cap  D}$ for every $I\in \tilde{\mathscr{R}}$.
Assume by contradiction that this was not the case. 
Therefore, there would exist $I\in  \tilde{\mathscr{R}}$ and  $p\in D$ such that $p\in \Omega(\psi_I)\backslash\overline{Per_h(\psi_I)}
=\Omega(\psi_I)\backslash\overline{Per(\psi_I)}$. 
The $C^1$-closing lemma for impulsive semiflows  (recall Corollary~\ref{thm:pugh}) guarantees that there  exists an impulse $J_1$, which is $C^1$-arbitrary close to $I$ and $\tilde{p}$ close to $p$  such that $\tilde{p}\in Per(\psi_{J_1})
$ and the closure of its orbit intersects the interior of $D$.
Using Lemma~\ref{abertura}, there exists a $C^1$-small perturbation $J_2\in \mathscr I^{\mathcal T}_{D,\hat D}$
 of $J_1$ such that 
$\tilde{p}$ is an hyperbolic period point for $\psi_{J_2}$.
%$p\in Per_h(\psi_{J_2})$. 
As $\mathscr I^{\mathcal T}_{D,\hat D}$ is a Baire space, $\tilde{\mathscr{R}}$ is dense in 
 $\mathscr I^{\mathcal T}_{D,\hat D}$. In particular,
$J_2$ can be arbitrarily $C^1$-approximated by an impulse $J_3 \in \tilde{\mathscr{R}}$ so that 
$\psi_{J_3}$ has a hyperbolic periodic point $\hat{p}$ arbitrarily close to $p$. 
This is in contradiction to the fact that $I$ is a continuity
point of $\Phi|_{\mathscr{R}_{per}}$.  
%
% By Lemma~\ref{cont.hip} there exists $\mathcal{U}$ a $C^1$-neighborhood of $\tilde
%{J}$  such that there exists $p_{\tilde{J}}$  hyperbolic periodic  close to $p$ for all $\tilde{J} \in \mathcal{U}$.  
%\marginpar{\tiny \color{red}Falta fechar a prova.}
The proof of Theorem~\ref{thmAA} is now complete.
\hfill $\square$

\section{Proof of Theorem~\ref{thmA}}\label{sec:equivrel}

%The proof of this theorem will follow from a semiconjugacy result for impulsive semiflows, of independent interest,
%together with Theorem~\ref{thmAA}.

Let $\varphi$ be a $C^1$-flow generated by  $X\in \mathfrak{X}^1(M)$ and $D$ be a smooth
%, \color{red} compact \color{black} 
submanifold of codimension one transversal to $X$
satisfying ~\eqref{eq:noncompact}. 
Given 
$I\in \mathscr I^{\mathcal T}_{D}$,
%$I\in \mathscr I_D$, 
the set $I(D)$ is a $C^1$-embedded submanifold (with boundary in case $\partial D\neq\emptyset$) so that 
$ 
I(D)\cap{D} =\emptyset$ and $I(D) \pitchfork X.$
Since these properties are preserved by $C^1$-small perturbations of the impulsive map $I$, there exists
a $C^1$-open neighborhood $\mathcal U_I\subset \mathscr I_{D}^{\mathcal T}$ of $I$ such that
\begin{equation}
\label{eq:vizUI}
J(D)\cap{D} =\emptyset \quad \text{and} \quad J(D) \pitchfork X, \qquad \forall J \in \mathcal U_I.    
\end{equation}
Diminishing $\mathcal U_I$, if necessary, we can take $r>0$ so $D\cap \bigcup_{\,t\,\in [-r,r]} \varphi_t(J(D))=\emptyset$ for every $J\in \mathcal U_I$.

In order to prove the theorem we now introduce an equivalence relation on the space $\mathcal U_I$
of impulses. Consider the equivalence relation $\sim$ in $\mathcal U_I$ given by:
$J_1\sim J_2$ if and only if there exists $r>0$ so that 
$$
J_j(D) \cap \bigcup_{\,t\,\in [-r,r]} \varphi_t(J_i(D)) = J_j(D)
\qquad \forall 1\le i, j \le 2.
$$
In rough terms, this property means that the space of finite orbits that intersect the cross-sections $J_1(D)$ and $J_2(D)$ coincide
(cf. Figure~\ref{figure00} below).

\begin{figure}[htb]
\begin{center}
\includegraphics[width=10cm,height=5cm]{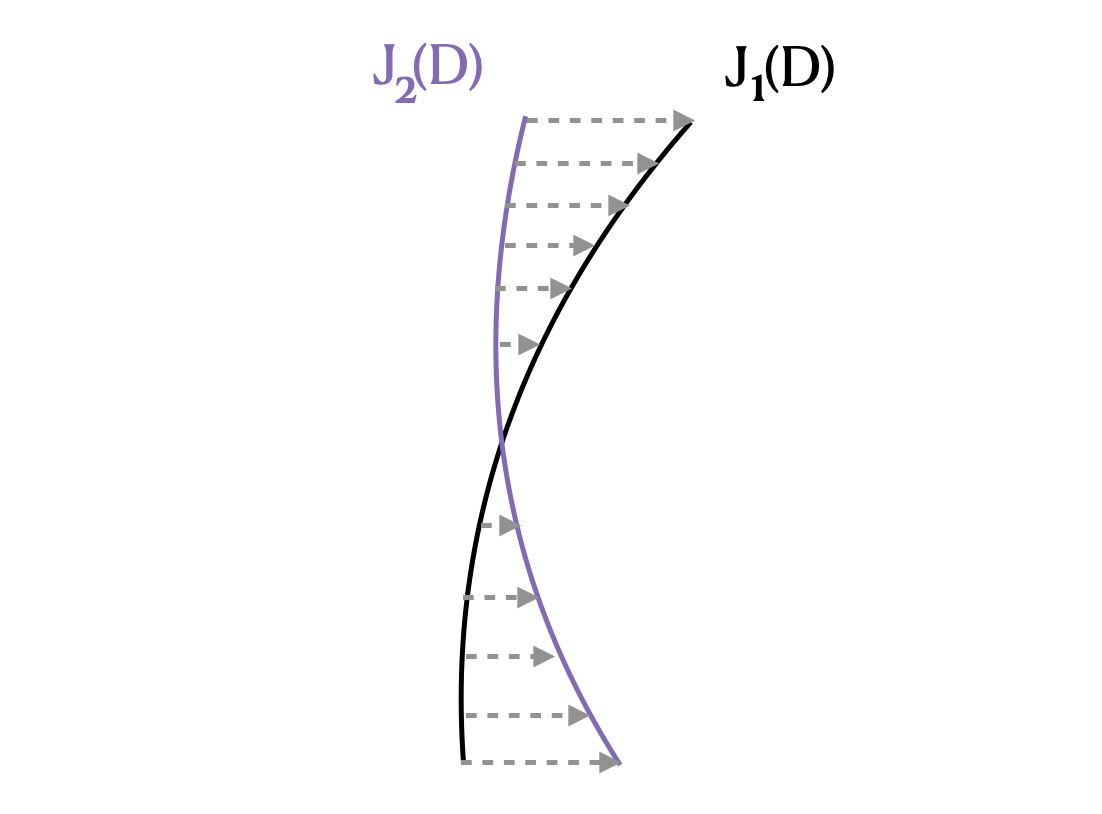}
  \caption{Representation of the images of impulses $J_1$ and $J_2$ in the same equivalence class, where the dotted arrows indicate orbits by the flow $\varphi$.}
\label{figure00}
\end{center}
\end{figure}

Moreover, %if $J_1\sim J_2$ 
if this is the case then the map
\begin{equation}\label{eq:zetat}
\begin{array}{rccc}
\zeta_{J_1,J_2} : &  J_1(D) & \to & J_2(D) \\
			  & x & \mapsto & \varphi_{\theta(x)} (x),
\end{array}
\end{equation}
%given by $\zeta_{J_1,J_2}(x)=\varphi_{\theta(x)}(x)$, 
where
$\theta(x)\in [-r,r]$ is uniquely defined by the property $\varphi_{\theta(x)}(x)\in J_2(D)$  and it 
is a $C^1$-diffeomorphism.
In particular:
\begin{itemize}
    \item[(i)] 
    if $J_1\in \mathscr I^{\mathcal T}_{D}$ and $J_2\sim J_1$
    then $J_2\in \mathscr I^{\mathcal T}_{D}$;
    \item[(ii)] $\mathcal U_I$ is partitioned in (uncountable many) %pathwise connected 
    equivalence classes $(\mathcal U_{I,\alpha})_{\alpha\in A}$. 
    %determined by $\sim$.
\end{itemize}
\medskip
Let us proceed with the proof of Theorem~\ref{thmA}. 
As the set  $\mathscr I^{\mathcal T}_D$  is separable, take a dense  and countable sequence $(I_j)_{j\ge 1}$ of impulses in $\mathscr I_D$.
Fix $j\ge 1$. Let $\mathcal U_j$ be an open neighborhood of the impulse $I_j$ and 
let $(\mathcal U_{j,\alpha})_{\alpha\in A_j}$ be
the partition of $\mathcal U_j$ on
equivalence classes. 
Item (i) above guarantees that either 
$\mathcal U_{j,\alpha}\subset \mathscr I^{\mathcal T}_D$ or $\mathcal U_{j,\alpha}\cap \mathscr I^{\mathcal T}_D=\emptyset$ for each $\alpha \in A_j$. Up to reduce the space of parameters $A_j$, if necessary, we can assume that $\mathcal U_{j,\alpha}\subset \mathscr I^{\mathcal T}_D$ for every $\alpha \in A_j$.
The next lemma characterizes the generic impulse inside of the equivalence classe.

\begin{lemma}\label{propclass}
For each $\alpha\in A_j$, there exists a Baire generic subset $\mathscr R_\alpha\subset \mathcal U_{j,\alpha}$ 
such that, for each $J\in \mathscr R_\alpha$,
\begin{equation}\label{eq:dchapeu2}
%\overline{Per_h(\psi_J) \cap J(D)} 
%= {\Omega(\psi_J) \cap J(D)}. 
\overline{Per_h(\psi_J)} \cap D = {\Omega(\psi_J) \cap D}. 
\end{equation}
\end{lemma}

\begin{proof}
Using that the image of the cross-section $D$ by impulses in the same equivalence class intersects the same finite pieces of orbits,  for each $\alpha\in A_j$ one can choose a $C^1$-smooth cross-section $\hat D_j$ so that 
$$
J(D) \cap \bigcup_{\,t\,\in [-r,r]} \varphi_t(\hat D_j) = J(D)
	\quad \text{and}\quad 
\hat D_j \cap \bigcup_{\,t\,\in [-r,r]} \varphi_t(J(D)) = \hat D_j
$$
for every $J\in \mathcal U_{j,\alpha}$.
By Theorem~\ref{thmAA}, there exists a $C^1$-Baire generic subset $\mathfrak R_\alpha\subset \mathscr{I}_{D, \hat D_j}$
so that 
%$
%\overline{Per_h(\psi_I) \cap \hat{D}_{\alpha}} = %{\Omega(\psi_I) \cap \hat{D}_{\alpha}}. 
%$ 
$\overline{Per_h(\psi_I) }\cap D = \Omega(\psi_I) \cap D$.
Consider the space
\begin{equation}
\label{eq:changecoord}
\mathscr R_\alpha:=\{ J \in \mathcal U_{j,\alpha} \colon \zeta_{J(D),\hat D_j}\circ J \in \mathfrak R_\alpha \},
\end{equation}
%and, by some abuse of notation,  denote by 
where $\zeta_{J(D),\hat D_j}$ is the holonomy map 
between the cross-sections $J(D)$ and $\hat {D}_j$ and $\theta(\cdot)$ is the corresponding hitting time function (recall ~\eqref{eq:zetat}).
Using that $J \mapsto \zeta_{J(D),\hat D_j}\circ J$ is a continuous, surjective and open map 
%\marginpar{\color{red} \tiny verificar pf}
we conclude that $\mathscr R_\alpha$
is a Baire generic subset of $\mathcal U_{j,\alpha}$. 

\smallskip
It remains to compare the orbits of the impulsive semiflows $\psi_J$ and 
$\psi_{\zeta_{J(D),\hat D_j}\circ J}$. Every periodic orbit $\gamma$ of period $T>0$ for the impulsive semiflow 
$\psi_{\zeta_{J(D),\hat D_j}\circ J}$ 
such that $\gamma\cap \hat D_j=\{x_1,x_2, \dots, x_k\}$ and $x_i\notin \partial \hat D_j$  is in correspondence with a periodic orbit  $\tilde \gamma$ for the impulsive semiflow $\psi_J$ such that  $\tilde \gamma\cap J(D)=\{\tilde x_1,\tilde  x_2, \dots, \tilde  x_k\}$, each $\tilde x_i$ does not belong to the boundary of $J(D)$, and of period $T+\sum_{i=1}^k \theta(\tilde x_i)$.
As the orbits of the impulsive semiflows $\psi_{\zeta_{J(D),\hat D_j}\circ J}$ and $\psi_J$ coincide in the complement of the set
$\bigcup_{t\in [-r,r]} \varphi_t(\hat D_j)$.  Therefore we have
$$
\Omega(\psi_J) \cap D = \Omega(\psi_{\zeta_{J(D),\hat D_j}\circ J})\cap D
\quad\text{and}\quad
\overline{Per_h(\psi_J)}\cap D = \overline{Per_h(\psi_{\zeta_{J(D),\hat D_j}\circ J})} \cap D.
$$
%${\Omega(\psi_J) \cap J(D)}$ is the image by 
%$\zeta_{J(D),\hat D_j}$
%of the set ${\Omega(\psi_{\zeta_{J(D),\hat D_j}
%\circ J}) \cap \hat D_j}$.
Altogether this proves that
$ 
\overline{Per_h(\psi_{\zeta_{J(D),\hat D_j\circ J}) }\cap  D_j} = {\Omega(\psi_{\zeta_{J(D),\hat D_j}\circ J}) \cap D_j},
$
for each $J\in \mathscr R_\alpha$, 
as desired.
%thus completing the proof of the proposition.
%$$
%\overline{Per_h(\psi_{\zeta_{J(D),\hat D_j}\circ J}) \cap \hat D_j} = {\Omega(\psi_{\zeta_{J(D),\hat D_j}\circ J}) \cap \hat D_j},
%$$
%for each $J\in \mathscr R_\alpha$ and, consequently, 
%$$
%\overline{Per_h(\psi_{J} \cap J(D))} = \Omega(\psi_J) 
%\cap J(D).
%$$
\end{proof}

\begin{remark}\label{rmk.continutypoints}
In view of the proof of Theorem~\ref{thmAA}, given $j\ge 1$ and $\alpha\in A_j$,
the $C^1$-Baire generic subset $\mathfrak R_\alpha\subset \mathscr{I}_{D, \hat D_j}$
consists of continuity points for the map
$$
\mathscr I_{D,\hat D_j} \ni J \mapsto \overline{Per_h(\psi_J)}\cap D.
$$ 
By construction, the $C^1$-Baire generic subset $\mathscr R_\alpha\subset \mathcal{U}_{j,\alpha}$
is formed by continuity points of the map
$ 
\mathcal{U}_{j,\alpha} \ni J \mapsto \overline{Per_h(\psi_J)}\cap D.
$
\end{remark}

%This proposition is instrumental to complete the proof of the theorem. Indeed, we obtain the following:

%\begin{corollary}
%There exists a $C^1$-Baire generic subset %$\mathscr R\subset 
%\mathscr I^{\mathcal T}_{D}
%$ 
%so that 
%$$\overline{Per_h(\psi_{J})} \cap D = \Omega(\psi_J) \cap D$$
%$$\overline{Per_h(\psi_{J}) \cap J(D)} 
%= \Omega(\psi_J) \cap J(D)$$ 
%for every $J\in \mathscr{R}$.
%\end{corollary}

%\begin{proof}
%As the set  $\mathscr I^{\mathcal T}_D$  is separable, take a dense sequence $(I_j)_{j\ge 1}$ of impulses in $\mathscr I_D$.
%For each $j\ge 1$, let $\mathcal U_j$ be an open neighborhood of the impulse $I_j$ and 
%let $(\mathcal U_{j,\alpha})_{\alpha\in A_j}$ be
%a partition of $\mathcal U_j$ on (uncountable) pathwise connected equivalence classes determined by $\sim$.
%By Proposition~\ref{propclass}, for each $j\ge 1$ and $\alpha\in A_j$ there exists a Baire generic subset $\mathscr R_{j,\alpha}\subset 
%\mathcal U_{j,\alpha}$ such that ~\eqref{eq:dchapeu2} holds for each $J\in \mathscr R_{j,\alpha}$.
%As the set $\bigcup_{\alpha\in A_j} \mathscr{R}_{j,\alpha}$ is a Baire generic subset of $\mathcal U_{j}$ and $\bigcup_{j\ge 1} \mathcal U_{j}=\mathscr I_D$ we conclude that
%$$
%\mathscr R:= \bigcup_{j\ge 1}\;  \bigcup_{\alpha\in A_j} \mathscr{R}_{j,\alpha}
%$$
%is a $C^1$ Baire generic subset of  $\mathscr I^{\mathcal T}_D$  satisfying the conclusions of the corollary.
%\end{proof}
We can now complete the proof of Theorem~\ref{thmA}. Note that $\bigcup_{j\ge 1} \mathcal U_{j}$ is a $C^1$-open and dense subset of $\mathscr I^{\mathcal T}_D$. Moreover, the map $\Gamma: \bigcup_{j\ge 1} \mathcal U_{j} \to D^*$
    defined by $\Gamma(J)=\overline{Per_h(\psi_J)} \cap D$ is  lower semicontinuous.
    Let $\mathscr R \subset \bigcup_{j\ge 1} \mathcal U_{j}$ be the $C^1$-Baire generic subset of continuity points of $\Gamma$.
    We claim that 
    \begin{equation}
    \label{eq:conclusionA}
    \overline{Per_h(\psi_J)} \cap D = {\Omega(\psi_J) \cap D} \qquad \text{for every}\;  J\in \mathscr R.    
    \end{equation}
    Indeed, any continuity point $J\in \mathcal{U}_{j}$ for $\Gamma$ is a continuity point for the restriction $\Gamma \mid_{\mathcal{U}_{j,\alpha}}$, where $\alpha\in A_j$ is uniquely determined by $J \in \mathcal{U}_{j,\alpha}$. By Remark~\ref{rmk.continutypoints} any such continuity point verifies equation~(\ref{eq:conclusionA}).
    This finishes the proof of the theorem.

%\color{red}
%Lemma~\ref{propclass} yields that 
%$$
%\mathscr D:= \bigcup_{j\ge 1}\;  \bigcup_{\alpha\in A_j} \mathscr{R}_{j,\alpha}
%$$
%is a $C^1$-dense subset of  $\mathscr I^{\mathcal T}_D$ so that $\overline{Per_h(\psi_J)} \cap D = {\Omega(\psi_J) \cap D}$ for every $J\in \mathscr D$.
\color{black}

\section{Proof of Corollary~\ref{corA}}\label{sec:corA}

Let $\varphi$ be the $C^1$-flow generated by $X\in \mathfrak{X}^1(M)$ and $D$ be a 
%\color{red} compact \color{black} 
smooth submanifold of codimension one transversal to $X$
satisfying ~\eqref{eq:noncompact}.

\medskip
Let us prove item (1). Assume that $I_0\in  \mathscr I^{\mathcal T}_D$ satisfies $\Omega(\varphi) \cap \partial D=\emptyset$. By compactness of $\Omega(\varphi)$ and $\partial D $ one gets 
$\delta_1:=\frac12 \mbox{dist}_H(\Omega(\varphi), \partial D)>0$ (recall  $\mbox{dist}_H$ stands for the Hausdorff distance). 
%\marginpar{\tiny \color{magenta} aqui retirei compacidade de D mas assumo (talvez isso seja a hipotese para retirar compacidade) que o fecho de D nao contem singularidades}
%By \color{red} compact \color{black}ness of $D$  and since $\varphi$ is a flow there exists s.t. %for every  $x \in D$ the flow-box centered at $x$% %has length $\delta_2$.
As $\varphi$  is a flow and $D$ satisfies \eqref{eq:noncompact},
%\color{red} $\overline D$ does not contain singularities \color{black} then 
there exists  $\delta_2 >0$ such that the flowbox $\{\varphi_t(D) \colon t\in [-\delta_2,0)\}$ does not intersect $D$. Hence, the trajectory of any point $x\in \Omega(\varphi)\cap \mbox{interior}(D)$ intersects $D$ infinitely many times. 
Taking $\delta= \min\{\delta_1, \delta_2 \}$  one can decompose the non-wandering set $\Omega(\varphi)$ of the original flow $\varphi$ in the two disjoint components
$$
\Omega_1(\varphi, D):=\{x\in \Omega(\varphi)\colon \varphi_t(x) \cap D \setminus B(\partial D, \delta) \neq \emptyset, \; \text{for some }\, t\in \mathbb R \} 
$$
$$
\Omega_2(\varphi, D):=\{x\in \Omega(\varphi)\colon \varphi_t(x)  \cap B(D,\delta) = \emptyset, \; \text{for every }\, t\in \mathbb R \}, 
$$
where $B(D,\delta)=\{x\in D \colon d(x,\partial D)<\delta\}$.
Since the set $\Omega(\varphi)$ is $\varphi$-invariant  then so it is each of the sets $\Omega_i(I,\varphi)$, $i=1,2$.
Theorem~\ref{thmA} implies that
there exists an open neighborhood $\mathcal V$ of $I_0$ and a Baire residual subset $\mathscr R\subset \mathcal V$ so that 
%the impulsive semiflow $\psi_I$ 
%determined by $I \in \mathscr R$ satisfies
$$
\overline{Per_h(\psi_I)} \cap D = {\Omega(\psi_I) \cap  D} 
\qquad \text{for every }\; I \in \mathscr R.
$$ 
%where ${Per_h(\psi_I)}$ denotes the set of hyperbolic periodic orbits of $\psi_{I}$. 
Now, observe that $\gamma_x(t)=\varphi_t(x)$ for every $t\in \mathbb R$ for all points $x$ whose trajectory does not meet $D$.
Therefore, by continuous dependence on initial conditions of the orbits of $\psi_I$ that do not intersect the boundary of $D$ 
one can write the $\Omega(\psi_I)$ as the (not necessarily disjoint) union of two invariant sets
$$
\Omega(\psi_I) =  \overline{Per_h(\psi_I)} \cup \, \Omega_2(\varphi, D).
$$
This implies that  $\Omega(\psi_I)\setminus D$ \color{black} is $\psi_I$-invariant, thus proving item (1).

%\
%\medskip
%Assume now that $\varphi$ is minimal and let $I \in \mathscr I_D$ be arbitrary. The minimality of $\varphi$ implies that for every $x\in M$ there exists $t\in \mathbb R$ so that $\gamma_x(t)\in D$. Therefore, there exist no points $x\in \Omega(\psi_I)$ whose orbit fails to intersect $I(D)$. 
%If $\mathscr R\subset \mathscr I_D$  is the $C^1$-Baire residual subset given by Theorem~\ref{thmA} then 
%$$
%\overline{Per_h(\psi_I) \cap I(D)} = {\Omega(\psi_I) \cap I( D)} \quad \forall I \in \mathscr R\ .
%$$ 
%This can be reformulated by saying that
%\begin{equation}\label{eq:porra}
%\overline{Per_h(\psi_I) } \cap D = {\Omega(\psi_I) \cap D} \quad \forall I \in \mathscr R\ .
%\end{equation}
%So, using Lemma~\ref{le:invOmega}, any $x\in \Omega(\psi_I)$ is so that $\gamma_{\tau_1(x)}(x) \in D\cap \Omega(\psi_I)$.
%Then, equation~\eqref{eq:porra} together with the continuous dependence on initial conditions implies that the periodic orbits of $\psi_I$ are dense in the non-wandering set $\Omega(\psi_I)$.
%Thus the set of hyperbolic periodic orbits is dense in $\Omega(\psi_I)$.
%This finishes the proof of the theorem.

\medskip
Let us now prove item (2). Assume that $\varphi$ is minimal and take 
$I_0 \in \mathscr I^{\mathcal{T}}_D$.
 The minimality of $\varphi$ implies that for every $x\in M$ there exists $t\in \mathbb R$ so that $\gamma_x(t)\in D$. This means that $0<\tau_1(x) <\infty$ for all $x\in M$. Therefore, there are no points $x\in \Omega(\psi_{I_0})$ whose orbit fails to intersect $D$. 
If 
$\mathscr R\subset \mathscr I^{\mathcal{T}}_D$
 is the $C^1$-Baire residual subset given by Theorem~\ref{thmA} then 
\begin{equation}
\label{eq:po}
\overline{Per_h(\psi_I) } \cap D = {\Omega(\psi_I) \cap D} \quad \forall I \in \mathscr R\ .
\end{equation}

In order to prove that the set of hyperbolic periodic orbits for the impulsive semiflow $\psi_I$ is dense in the non-wandering set
$\Omega(\psi_I)$ it is enough to show these are dense in $\Omega(\psi_I)\setminus D$.
Fix $z \in \Omega(\psi_{I})\setminus D$, and assume without loss of generality that $\tau_{-1}(z)<\infty$
where 
$$
\tau_{I,-1}(x) = 
\inf\{t\ge 0 \colon \varphi_{-t}(x) \in I(D) \}
$$
is the first hitting time of a point $x$ to $I(D)$ by the reverse-time flow $(\varphi_{-t})_{t \in \mathbb R}$
(by minimality of $\varphi$ the forward images of points of $I(D)$ by the semiflow $\psi_I$ are dense in $\Omega(\psi_I)$).  
As the impulse $I$ is invertible, the same argument used in the proof of 
Lemma~\ref{le:invOmega} ensures that 
$$
\gamma_{t}(z) \in I(D)\cap \Omega(\psi) \qquad \forall t\in [\tau_{I,-1}(z),\tau_1(z)]
$$
Equation ~\eqref{eq:po} implies that there exists a sequence of periodic points $(p_n)_{n\ge 1}$ for $\psi_I$ on $I(D)$ such that
$d(p_n, \varphi_{-\tau_{I,-1}(z)}(z)) \le d_{I(D)}(p_n, \varphi_{-\tau_{I,-1}(z)}(z) ) \to 0$ as $n$ tends to infinity. Then, the continuous dependence of the initial flow $\varphi$ on initial conditions ensures that the sequence of hyperbolic periodic points
$(\varphi_{\tau_{I,-1}(z)}(p_n))_{n\ge 1}$ are such that
$
d( \varphi_{\tau_{I,-1}(z)}(p_n), z ) \to 0
$
as $n\to\infty$. 
Thus the set of hyperbolic periodic orbits is dense in $\Omega(\psi_I)$. This proves item (2) and  finishes the proof of Theorem~\ref{thmAA}.

\color{black}

\hfill $\square$

\section{Examples}\label{sec:examples}

In this section we shall provide some examples which illustrate both the dependence of the non-wandering set
of impulsive semiflows as a function of the impulse, and that part of the non-wandering set may not be affected by perturbations of the impulse (as the later are determined by the impulsive region, and their image). 
\color{black}

\begin{example}\label{ex:ACrobusto} (Robust non-invariance of the non-wandering set)
Consider the impulsive semiflow constructed on the annulus written in polar coordinates as
 $$M=\left\{(r\cos\theta,r\sin\theta)\in\R^2:  1\le r\le 2  , \,\theta \in [0,2\pi]\right\}
 $$ 
 constructed in \cite[Example~2.1]{AC14}. 
Let $\varphi$ be the flow generated by the vector field $X(r,\theta)=(0,1)$
%  $$
%   \begin{cases}
%   r'=0 & \\
%   \theta{\, '}=1 .&
%   \end{cases}
%$$
whose trajectories are circles spinning counterclockwise. Consider the local cross-sections
 $$
 D=\{(r,0)\in M: 1\le r \le 2\}\quad \text{and} \quad\hat D=\{(r,0)\in M: -\frac32\le r \le -1\}
 $$ 
 to the flow $\varphi$
and define $I_0:D\to \hat D$ by
 $I_0(r,0)=\left(-\frac12-\frac12r, 0\right).$
Any $C^1$-small perturbation of $I\in \mathscr I_{D,\hat D}$ is such that 
$I(1,0)=(-1,0)$, $I(2,0)=(-\frac32,0)$ and $\|DI(r,0)\|<1$ for every $r\in [1,2]$.
In consequence, it is not hard to check that for every impulse $I$ that is $C^1$-close to $I_0$
one has that
$$\Omega(\psi_{I})=\left\{(\cos\theta, \sin\theta): \frac{3\pi}{2}\le \theta\le 2\pi\right \},$$
thus it is not forward invariant by $\psi_{I}$ (the forward trajectory of $(1,0)\in\Omega(\psi_{I})$ is clearly not contained in $\Omega(\psi_{I})$). 
This example shows that the non-wandering set is not $\psi_I$-invariant for a $C^1$-open set of impulses $I$.
\end{example}

%\begin{example}\label{ex:bonitinho}
%This example shows......
%\marginpar{\tiny \color{red} falta}
%\end{example}

The next example illustrates that it might exist part of the non-wandering set for impulsive dynamics which remains unaltered by arbitrary perturbations of the impulse.

\begin{example}\label{ex:prey} (Non-wandering set of the original flow and impulsive semiflow) \label{ex: periodico fora do impulsivo}
Let $X$ be a vector field of a predator-prey model given by
$$
 \left\{
\begin{array}{rcl} 
\dot{x}&=&  x(3-x-y)\\
\dot{y}&=&   y (-1+x-y)
\end{array}
\right.
$$
and let $\varphi$ be the flow generated by $X$ on the compact surface $M= [0,4]\times[0,2]$
with boundary, pointing inwards.  All the trajectories of $\varphi$ converge to the fixed point $p=(2,1)$ in positive as time tends to $+\infty$.
The sets $D= \!\{1\}\! \times \![0,2]$, $\hat{D}= \{1/2\} \times[0,2]$ 
are cross-sections to the flow $\varphi$. 
%\
%We refer the reader to \cite{??} for more details.
%\color{black}
%
Define $I: D \mapsto \hat D $  by $I(x)= (x- 1/2 ,y/2)$ and consider  the impulsive semiflow  $\psi_I$ generated by $(M, \varphi, D,I)$. The  non-wandering set of  $\psi_I$ has two components, namely, the atracting fixed point $p$ and the segment $[1/2, 1]\times\{ 0\}$.  It is not hard to check that  for any $C^1$-small perturbation $J: D\to \hat D$ of the impulsive map $I$, the non-wandering set of $\psi_J$ coincides with the non-wandering set of $\psi_I$. 
%fixed point $p$ will still belong to the non-wandering set of $\psi_J$.  In other words, 

\end{example}

In the next example we observe that even though an impulse $I \in \mathscr I_{D,\hat D}$ is a $C^1$-diffeomorphism it may occur that $I(\Omega(\psi_I)\cap D) \neq %\Omega(\psi_I) \cap I(D)$.
\Omega(\psi_I) \cap \hat D$.

\begin{example} (Non-wandering set on the domain and range of the impulse)
    \label{ex:TBA}

Consider the impulsive semiflow constructed on the disk written in polar coordinates as
 $$
S_2=\left\{(r\cos\theta,r\sin\theta)\in\R^2:  0\le r\le 2  , \,\theta \in [0,2\pi]\right\}
 $$ 
and let $\varphi$ be the flow generated by the vector field $X(r,\theta)=(0,1)$
whose trajectories are either constant at the origin or circles spinning counterclockwise.
One can extend this flow  to the disk  of radius 10, $S_{10}$, in a way that no new periodic orbits are created. Notice that none of the periodic orbits can be hyperbolic because they are not isolated.  
 Consider the position of the local cross-sections $D,\hat D$ as in Figure~3.
\begin{figure}[htb]\label{fig3}
\begin{center}
\includegraphics[width=10cm,height=5cm]{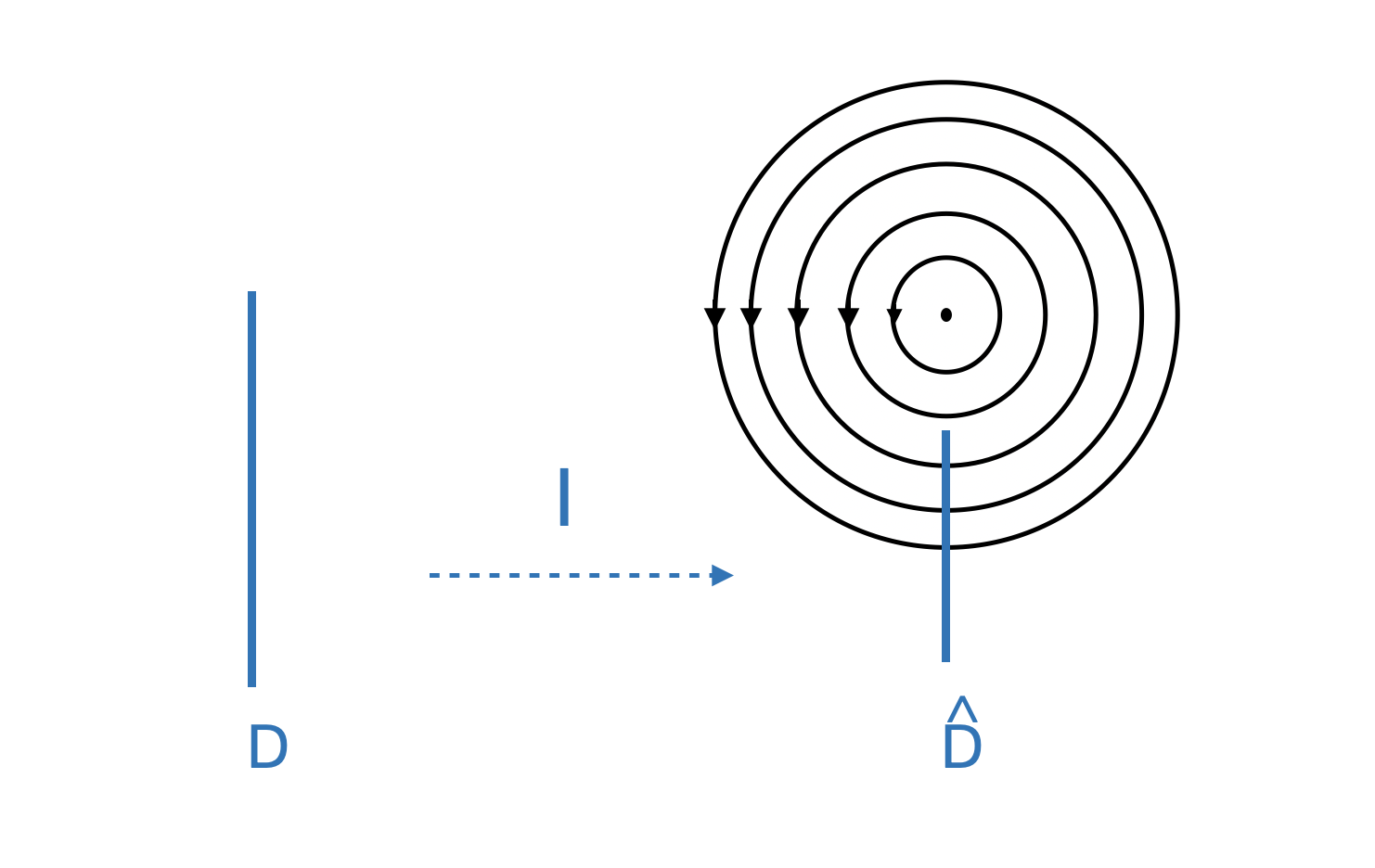}
  \caption{Non-hyperbolic periodic orbits in the set $\Omega(\psi_I) \cap \hat D$}
\end{center}
\end{figure}
It follows from Theorem~\ref{thmAA} that 
there exists a Baire residual subset 
 $\mathscr R
\subset \mathscr I_D^{\mathcal T}$
such that 
$  
\overline{Per_h(\psi_I)} \cap D = {\Omega(\psi_I) \cap D},
 $ 
for every $I \in \mathscr R$.
However, despite the fact that each impulse is assumed to be a $C^1$-diffeomorphism, 
$$
Per_h(\psi_I) \cap \hat D \neq {\Omega(\psi_I) \cap \hat D} \qquad \text{for every $I\in \mathscr I_{D,\hat D}^{\mathcal T}$,}
$$ 
as ${\Omega(\psi_I) \cap \hat D}$ contains an open set of points whose orbits are periodic and non-hyperbolic. In fact, 
$I(\Omega(\psi_I) \cap D) \subsetneq \Omega(\psi_I) \cap \hat D$ 
for every $I\in 
\mathscr I_{D,\hat D}^{\mathcal T}.$
\end{example}
\color{black}

Corollary~\ref{corA} provides a sufficient condition, involving the boundary of the impulsive region, for the denseness of periodic orbits
among the part of the non-wandering set that intersects the impulsive region. In the next example we illustrate that many of such examples can be easily constructed using suspension flows. 

\begin{example}\label{ex:bordo}
Let $f \in \text{Diff}^{\; 1} (M)$ be a $C^1$-diffeomorphism on a  compact manifold $M$ such that 
$\Omega(f)$ is not connected. 
%Thus, $\Omega(f)$ can not be written as the union of two non-empty sets such that each one is disjoint from the closure of the other. 
Thus, there exists an open set $O\subset M$ for which $O\cap \Omega(f)\neq\emptyset \neq 
O^c \cap \Omega(f)$
and $\partial O \cap  \Omega(f)=\emptyset$.
Given a $C^1$-roof function $r:M\rightarrow(0,\infty)$ consider the quotient space
\begin{equation}
M^{r}=\Big\{ (x,t) \in M \times \mathbb R_0^+  \colon 0\le t\le r(x),x\in M\Big\} /\sim\label{eq:suspflow}
\end{equation}
where $(x,r(x))\sim(f(x),0)$.
The \emph{suspension flow over $f$ with height function $r$} is
the flow $(\varphi_{t})_{t\in\mathbb{R}}$ in $M^{r}$ defined by
$$
\varphi_{t} (x,s)=\Big(f^{n}(x),s+t- \sum_{i=0}^{n-1}r(f^{i}(x))\Big), \quad  \text{for $t>0$,} 
$$
where $n \ge 0$ is uniquely determined by 
$
\sum_{i=0}^{n-1}r(f^{i}(x)) \le t+s < \sum_{i=0}^{n}r(f^{i}(x)). 
$
It is simple to check that the non-wandering set of the flow $\varphi$ is the saturated set
$$
\Omega(\varphi)= \bigcup_{t\in \mathbb R} \; \varphi_t ( \Omega(f) \times \{0\} ).
$$
Define $\hat D=\overline O \times \{0\}$ (in local coordinates), the boundary $\partial \hat D$ 
does not intersect the boundary of $\Omega(\varphi)$.  
Then, Corollary~\ref{corA} implies that if $D$ is a local cross section to the flow and $I_0\in \mathscr I_{D,\hat D}$ then there exists $\delta>0$ and
an open neighborhood $\mathcal V$ of $I_0$ and  a Baire residual subset $\mathscr R\subset \mathcal V$ so that, for every   $I\in \mathscr R$
one can write the non-wandering set $\Omega(\psi_I)$ as a (possibly non-disjoint) union
$$
\Omega(\psi_I) =  \overline{Per_h(\psi_I)} \cup \, \Omega_2(\varphi, D),
$$
where the set $\Omega_2(\varphi, D)\subset \Omega(\psi_I)$ is  $\varphi$-invariant 
and does not intersect a $\delta$-neighbor\-hood of the cross-section $I(D)$.

\end{example}

The next two classes of examples suggest that one cannot expect any type of semi-continuity of the non-wandering set, 
when one considers the original flow as some limit of impulsive semiflows. Let us be more precise.
\color{black}

\begin{example}\label{ex:MorseSmale}(Explosion of the non-wandering set)
Let $M\subset \mathbb R^2$ be the disk of radius 3 centered  at the origin, let $X$ be the radial vector field  on $M$ 
pointing inwards the disk and let $\varphi$ be the smooth flow generated by it. Note that the non-wandering set 
is given by $\Omega(\varphi)=\{0\}$. Take 
$$
D= \Big\{x=( r, \theta) \in M: r=1 , \,  \theta \in [0,2\pi] \Big\}
$$ 
and 
for each $\delta \geq 0$ define the impulse map $I_\delta: D \to M$  by $I_\delta(x)= ( (1+\delta)r, \theta)$ for all $x\in D$.  Notice that if 
$\psi_{I_\delta}$ 
is the impulsive semiflow generated by  $(\varphi, M, D, I_\delta)$ then
$$
\Omega(\psi_{I_\delta}) = \{0\} \; \cup \; 
	\big\{( r, \theta): 1\leq r \leq 1+\delta , \, \theta \in [0, 2\pi]  \big\}
$$
and so $\text{dist}_H(\Omega(\psi_{I_0}),\; \Omega(\psi_{I_\delta}))=1$ for every $\delta>0$ (here $\text{dist}_H(\cdot, \cdot)$ stands for the usual Hausdorff distance).
%It is not hard to check that the impulsive non-wandering set is the origin union with the annulus   
%$A= \{( r, \theta): 1\leq r \leq 2 , \, \theta \in [0, 2\pi]  \}$. 
\color{black}

\end{example}

\begin{example}\label{ex:toro} (Implosion of the non-wandering set)
Consider a non-singular $C^1$-smooth flow $\varphi$ on $\mathbb T^2$ and let $D$ be a smooth global cross-section, diffeomorphic to $\mathbb S^1$. We claim that there exists a Baire residual subset $\mathscr R\subset \mathscr I_D^{\mathcal{T}}$ so that, for every $I\in \mathscr R$, the set $\Omega(\psi_I) \setminus D$ is a $\psi_I$-invariant
proper subset of $\mathbb T^2$.
The invariance follows directly from Corollary \ref{corA}.  
On the other hand, as for each $I \in \mathscr I_D^{\mathcal{T}}$ the image $I(D)$ is a embedded circle transversal to the flow direction, all points in an open cilynder defined by $D$ and $I(D)$
are wandering. In particular, even if $\varphi$ is a minimal flow on the torus these impulsive semiflows are not transitive. 
\end{example}

We now consider an example where impulsive dynamics is motivated by physical external action
over two-dimensional billiards.

\begin{example}(Impulsive billiard flows)\label{ex:billiards}
A planar billiard flow is a dynamical system describing the motion of a point particle moving freely 
in the interior of a connected compact subset $S\subset \mathbb R^2$ with piecewise smooth boundary $\partial S$, where a particle flows along straight lines until it hits the boundary of $S$. If the collision occurs at a smooth boundary point, the particle gets reflected in a way that
the tangential component of the particle velocity remains the same while the normal component
changes its sign 
%but it is physically relevant to %consider the case where the %reflection angle on the incidence %angle is either a contraction or a %dilation
(see e.g. \cite{Magno} and references therein). 
In this way,
the billiard flow $\phi$ can be modelled by a suspension flow
over a Poincar\'e map
%\marginpar{\tiny \color{magenta} secao nao eh compacta}
$$
P: \partial S \times (-\frac\pi2,\frac\pi2) \to \partial S \times (-\frac\pi2,\frac\pi2)
$$
which associates to a pair $(x,\theta)$
a pair $(x',\theta{\, '})$ where: (i) the point $x'$ stands for the position of the first collision with $\partial S$ of the trajectory starting at the point $x$ and making an angle $\theta$ with $(T_x \partial S)^\perp$; (ii) the angle 
$\theta{\, '}$ stands for the angle between the coliding trajectory and 
$(T_{x'} \partial S)^\perp$.
The roof function of the suspension flow 
$r: \partial S \times (-\frac\pi2,\frac\pi2) \to \mathbb R_+$ is such that $r(x,\theta)$ represents the flight time of the trajectory determined by the point $(x,\theta)$ until the next collision to $\partial S$. 
The billiard flow $\phi : \mathcal M \to \mathcal M$ is the suspension flow $\phi=(\phi_t)_{t}$ written in local coordinates by 
$\phi_t((x,\theta),s)=
((x,\theta),s+t)$ on the space 
$$
\mathcal M=\Big\{((x,\theta),t) \in \partial S  \times  (-\frac\pi2,\frac\pi2) \times \mathbb R \colon 0\le t \le r(x,\theta)\Big\} / \sim
$$
where $((x,\theta),r(x,\theta))\sim 
(P(x,\theta),0)$.

Let $\phi$ be the billiard flow on the disk $S=D(0;1)\subset \mathbb R^2$ of radius 1 centered at the origin. It is well known that 
$
P(x,\theta)=(R_{\pi-2\theta}(x),\theta)
$ 
where $R_\alpha$ stands for the rotation of angle $\alpha$ on the circle and that the length of the line segment defined by the points $x$ and $R_{\pi-2\theta}(x)$ is $2\cos \theta$
(see e.g. \cite[page 6]{Park}). 
In consequence, using basic trigonometry and topological features of rotations: 
\begin{enumerate}
    \item[(i)] the periodic orbits of 
    $\phi$ are dense in the non-wandering set $\Omega(\phi)=\mathcal M$;
    \item[(ii)] $(x,\theta)$ has a $\phi$-periodic orbit if and only if $\frac{\theta}{\pi}\in\mathbb Q$;
    \item[(iii)] if $\theta=\frac{p}{q} \pi$ with $p,q\in \mathbb Z$ are coprime then every point $(x,\theta)$ is $\phi$-periodic with period $2q\cos \theta$ (hence all periodic points are not hyperbolic);
    \item[(iv)] if $\theta \notin \pi \mathbb Q$ then every point $(x,\theta)$ is dense in its omega-limit set; 
    \item[(v)] the non-wandering set of $\phi$ coincides with the whole space $\mathcal M$. 
\end{enumerate}

Take two disjoint smooth arcs $D_0, \hat D_0\subset \partial S=\mathbb S^1$
and the local cross-sections to the flow $D=D_0\times (-\frac{\pi}{2},\frac{\pi}{2})$ and $\hat D=\hat D_0\times (-\frac{\pi}{2},\frac{\pi}{2})$. If $\theta \notin \pi \mathbb Q$
then the set $D_0 \times \{\theta\}$
is partitioned in at most a countable set of domains where $\tau$ is locally constant, hence
$$
\sup_{(x,\theta) \in \hat D, \; \theta \, \notin \, \pi \mathbb Q} \Big|\frac{d\tau_1}{dx}(x)\Big| = 0.
$$
Similarly, 
$$
\sup_{(x,\theta) \in \hat D, \tau_1(x,\theta) <\infty , \; \theta \, \in \, \pi \mathbb Q} \Big|\frac{d\tau_1}{dx}(x)\Big| = 0.
$$
Therefore, by Theorem~\ref{thmAA} there exists a Baire residual subset $\mathscr R \subset 
\mathscr I^{\mathcal T}_{D,\hat D}
$
such that the impulsive billiard flow $\psi_I$ 
determined by $I \in \mathscr R$ has hyperbolic periodic orbits and that  
$\overline{Per_h(\psi_I)} \cap D = {\Omega(\psi_I) \cap D}$.
\end{example}

Let us now consider the following class of impulsive Anosov flows. 

\begin{example} (Impulsive Anosov flows)
    Let $X$ be a $C^1$-vector field generating a 
    $C^1$ codimension one Anosov flow $\phi$ on a  compact  manifold $M$ which admits a global cross-section, i.e. an Anosov flow which admits a smooth codimension one submanifold $D$
    (see e.g. \cite{Ghys,Plante} for several partial classifications of such classes of flows).
The set of vector fields generating such an Anosov flow form a $C^1$-open subset of $\mathscr{X}^r(M)$, and: (i) the Poincar\'e map $F: D \to D$ is uniformly hyperbolic, and (ii) the flow $\phi$ has a dense set of periodic orbits. Moreover,  
    if $t>0$ is chosen small so that the global cross section $\hat D:=\phi_t(D)$ does not intersect $D$ then the map $\tau_1\mid_{\hat D}$ is $C^1$-smooth, thus
   $$
\sup_{x \in \hat D} \Big|\frac{d\tau_1}{dx}(x)\Big| <\infty
$$
and so $\mathscr I^{\mathcal{T}}_{D,\hat D} \neq \emptyset$.
While the Poincar\'e map $P_I= I\circ \phi_{\tau(\cdot)}$
is seldom hyperbolic, Theorem~\ref{thmA} implies that periodic orbits are dense in the attractor of the impulsive Anosov flows 
for a $C^1$-Baire generic set of impulses $I$
in $\mathscr I_{D,\hat D}$. 
    
%For each $I\in \mathscr I^{\mathcal{T}}_{D,\hat D}$ consider the Anosov impulsive semiflow $\psi_I$ generated by $(M, \varphi, D, I)$.
%As no a priori assumptions are given 
%on the impulses $I$, it may occur that the Poincar\'e map $P_I$ 
%does not admit a dominated splitting, and carrries no hyperbolicity. Nevertheless,
\end{example}

\color{black}

Finally, we finish this section by considering a class of impulsive Lorenz attractors, and prove that even though the impulsive semiflow may not 
exhibit partial hyperbolicity one can still prove that for a typical impulsive Lorenz  the periodic orbits are dense 
in the non-wandering set. 
\color{black}

\begin{example}(Impulsive flows derived from geometric Lorenz attractors)\label{ex:Lorenz}
%The geometric Lorenz attractors were introduced %independently by Guckenheimer and Williams~\cite{gw} 
%and Afra$\breve {\rm \i}$movi$\check{\rm c}$, Bykov and %Sil'nikov, \cite{abs}
%for vector fields on a closed 3-manifold $M^3$. 
	%We say that 
 A vector field $X\in\mathscr{X}^r(M^3)$ ($r\geq1$) exhibits a \emph{geometric Lorenz attractor} $\Lambda$, if $X$ has an attracting region $U\subset M^3$ such that  $\Lambda=\bigcap_{t>0}\phi^X_t(U)$ is a singular hyperbolic attractor and satisfies:
	\begin{itemize}
		\item[(i)] There exists a unique singularity $\sigma\in\Lambda$ with three exponents $\lambda_1<\lambda_2<0<\lambda_3$, which satisfy $\lambda_1+\lambda_3<0$ and $\lambda_2+\lambda_3>0$, whose eigenspaces (in local coordinates $x_1, x_2, x_3$ in $\mathbb R^3$) are identified with the canonical axis.  
		\item[(ii)] $\Lambda$ admits a $C^r$-smooth cross section which in local coordinates can be written as
		$\Sigma=[-1,1]  \times [-1,1]  \times \{1\}$ and
		for every $z\in U\setminus W^s_{\it loc}(\sigma)$, there exists $t>0$ such that $\phi_t^X(z)\in\Sigma$
		(here $W^s_{\it loc}(\sigma)$ stands for the stable manifold of the hyperbolic singularity $\sigma$)
		\item[(iii)] With the identification
		of $\Sigma$ with $[-1,1]\times[-1,1]$ (by a $C^1$-diffeomorphism) where $l=\{0\}\times[-1,1]=W^s_{\it loc}(\sigma)\cap\Sigma$, 
		the Poincar\'e map $P:\Sigma\setminus l\rightarrow\Sigma$ is a skew-product map
		$$
		P(x,y)=\big( f(x)~,~H(x,y) \big), \qquad \forall(x,y)\in[-1,1]^2\setminus l
		$$
		where
		\begin{itemize}
			\item $H(x,y)<0$ for $x>0$, and $H(x,y)>0$ for $x<0$; \color{black} 
			\item $\sup_{(x,y)\in\Sigma\setminus l}\big|\partial H(x,y)/\partial y\big|<1$, 
			and
			$\sup_{(x,y)\in\Sigma\setminus l}\big|\partial H(x,y)/\partial x\big|<1$; \color{black} 
			\item the one-dimensional quotient map $f:[-1,1]\setminus\{0\}\rightarrow[-1,1]$ is $C^1$-smooth and satisfies
			$\lim_{x\rightarrow0^-}f(x)=1$,
			$\lim_{x\rightarrow0^+}f(x)=-1$, $-1<f(x)<1$ and
			$f'(x)>\sqrt{2}$ for every $x\in[-1,1]\setminus\{0\}$.
		\end{itemize}
	\end{itemize}

\begin{figure}[htbp]
	\centering
	\includegraphics[scale=0.16]{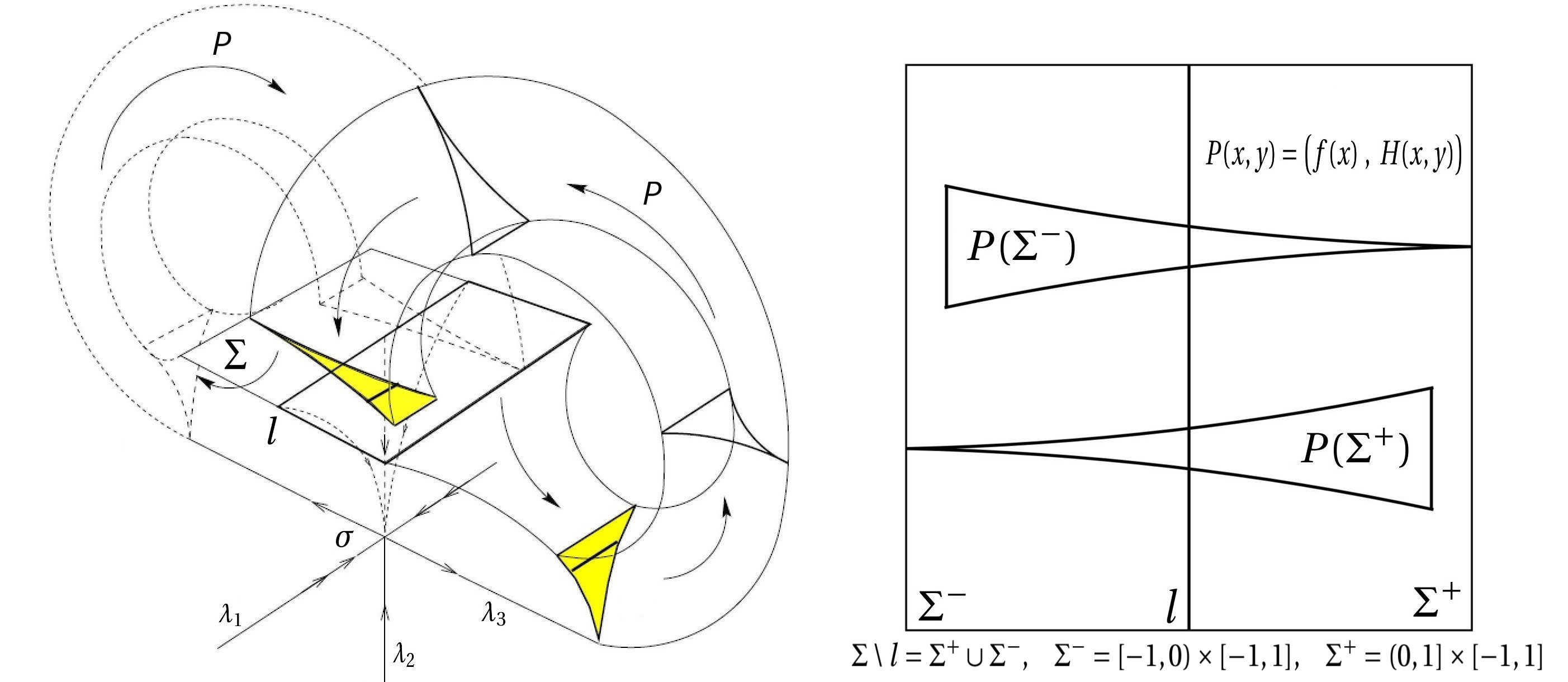}
	\caption{\small Geometric Lorenz attractor and return map: the set $\hat D$ formed by the two triangles with cusps located at $|x_1|=1$
	(on the left), 
	the cross-section $D=\Sigma\setminus l$ (on the right, up to identification)}
\end{figure}

The set of vector fields exhibiting a geometric Lorenz attractor forms a $C^1$-open set in $\mathscr{X}^r(M)$ (see e.g.~\cite[Proposition 4.7]{STW}). Moreover, each geometric Lorenz attractor is a singular-hyperbolic homoclinic class, the map $P$ has a dominated splitting and the flow has a dense set of periodic orbits
(cf. ~\cite{AP} for definitions and proofs).
%\begin{definition}\label{Def:homoclinic-class}
%	Let $X\in\mathscr{X}^1(M)$ and assume $p,q$ are two hyperbolic periodic points of $X$.
%	The {\it homoclinic class} of $p$ is defined as
%	$$H(p)=\overline{W^{s}(orb(p)) \pitchfork W^{u}(orb(p))},$$ 
%	that is, the closure of the set of transversal intersections between the stable and  unstable manifolds of the periodic orbit $orb(p)$ of $p$.  
%	The two hyperbolic periodic orbits $orb(p)$ and $orb(q)$ are {\it homoclinically related} if the stable manifold  of $orb(p)$ has non-empty transverse intersections with the unstable manifold of $orb(q)$ and vice versa. 
%	A homoclinic class is  called \emph{non-trivial} if it is not reduced to a single hyperbolic periodic orbit.
%\end{definition}

\smallskip
The Poincar\'e map $P$ of the geometric Lorenz attractors can be written as a composition $P=P_2\circ P_1$, where 
$P_1: D \to \hat D$ and $P_2: \hat D \to D$ are Dulac maps, i.e.
the first hitting time maps for the flow $(\phi^X_t)$ between $D:=\Sigma\setminus l$ and $\hat D=\{\phi^X_{h(x)}(x) \colon x \in D\}$, where 
$h(x)=\inf\{t>0 \colon \phi^X_{h(x)}(x) \in \{|x_1|=1\}\}$. 
Moreover, the function $\tau_1$ is piecewise $C^1$, and the supremum of its derivative is bounded in the two domains of smoothness, hence 
$$
\sup_{x\in \hat D} \Big|\frac{d\tau_1}{dx}(x)\Big| <+\infty.
$$

In this way, the impulsive semiflow determined by the flow $(\phi_t^X)_t$ and the impulse $I=P_2$ encloses a large part of the information
relative to the geometric Lorenz attractor. Our main results yield more interesting results in the case the impulse $I$ is such that $I\circ P_1$ 
does not admit a dominated splitting or fails to preserve the (vertical) stable foliation for the original flow.
Indeed, even if this is not the case, 
%We now consider impulsive semiflows derived from Lorenz attractors. Consider the two triangular shaped impulsive region $D$
%and $\hat D=\Sigma$. If an impulse $I: D\to \hat D$ is $C^1$-close to the Poincar\'e map defined by the 
%original flow then the impulsive Lorenz attractor is a $C^1$-pertubation, hence periodic orbits are dense. 
Theorem~\ref{thmAA} implies that 
there exists a $C^1$-Baire generic set of impulses whose hyperbolic periodic orbits are dense in the component of the non-wandering set that intersects $D$. 
\end{example}

\begin{remark}\label{rmkLorenz}
    It is worth mentioning that the choice of $D$ and $\hat D$ above cannot be interchanged. Indeed, if this was the case the hitting time $\tau_1$ would be piecewise $C^1$-smooth but its supremum would be infinite.  
\end{remark}

\color{black}
%%%%%%%%%
\section{Concluding remarks and future perspectives}\label{sec:perspectives}

In this final section we shall comment on the assumptions of the main results and describe some of the possible directions
of further study initiated by this approach. 

\subsubsection*{Non-invertible impulse maps}
In this paper we considered impulsive semiflows where the orbits are determined by an initial flow, generated by a smooth vector field, together with an impulsive map, assumed to be an embedding, acting on an impulsive region $D$ which is a local cross-section to the original flow. The invertibility of the impulse maps is used in our proof of the %not crucial to obtain the 
differentiability of the Poincar\'e maps %\textcolor{red}{Agora usamos fortemente! Precisamos reescrever}
but it is not used in the study of the hyperbolicity of periodic orbits for the impulsive semiflow (recall
Subsections~\ref{subsec:Poincare} and ~\ref{subsec:hyp}). Moreover, the invertibility is used implicitly to control the geometry of the discontinuity sets, 
in the fact that $I\in \mathscr I_{D,\hat D}$ maps $\partial D$ to $\partial \hat D$, and in the proof of the closing lemma (Theorem~\ref{thm:pugh}). Indeed, as even in the context of maps, the closing lemma seems not to be completely understood for maps with discontinuities and/or critical points, it would be interesting to determine whether the statement of Theorem~\ref{thmA} could be extended for semiflows associated to non-invertible impulsive maps.

\subsubsection*{Volume preserving impulsive semiflows}
Several important classes of (semi)flows are known to preserve volume. In case the impulsive map preserves volume it is not hard to check that the resulting
impulsive semiflow also preserves the volume, and it is natural to ask if some result analogous to Theorem~\ref{thmA} can be obtained in this context.  
Even though we believe this can be the case, where the key closing lemma  (Theorem~\ref{thm:pugh}) should be refined so that the perturbed maps preserve volume, in the spirit of \cite{Pugh2}, we will not consider this question here.

\subsubsection*{Ergodic closing lemma}
In this paper we prove the invariance of the non-wandering set for $C^1$-typical impulsive semiflows among certain classes of checkable conditions
by proving the abundance of periodic orbits on the part of the non-wandering set which intersects the impulsive region (recall Theorem~\ref{thmA}
and Corollary~\ref{corA}).  An immediate consequence is that such impulsive semiflows admit invariant measures, in particular those supported on 
the hyperbolic periodic orbits. A different problem concerns the denseness of ergodic probability measures in the space of invariant probability measures, 
a problem which in the context of diffeomorphisms follows from Ma\~n\'e's  ergodic closing lemma \cite{Man}.
While the denseness of the ergodic probability measures in the space of invariant measures  cannot be expected to be typical in general (as impulses may not affect part of the non-wandering set, as illustrated by Example~\ref{ex:prey}), it is an open question to determine whether this holds among the space of invariant probability measures whose support intersect the impulsive region.  
\color{black}

\subsubsection*{$C^0$-perturbations }
The proofs of the main results use strongly the perturbation theory in the $C^1$-topology on the space of impulsive maps, and the fact that hyperbolicity
of periodic points is robust by $C^1$-perturbations of the impulse. While it is not known if the closing lemma can be extended to stronger topologies $C^r$, 
$r\ge 2$, we can obtain a counterpart of the main result in the context of impulsive semiflows determined by continuous impulses, endowed with the 
$C^0$-topology. The arguments, inspired by 
%\cite{BTV,BTV2,ZG}
explore covering relations and build over the fact that such typical impulsive semiflows 
satisfy the shadowing property and a reparameterized gluing orbit property, conditions which  are known to have several relevant topological and ergodic 
features (see e.g. \cite{BTV2,BoTV} and references therein). 
This picture is completed in the forthcoming paper \cite{STV}.

\subsubsection*{Impulsive semiflows parameterized by vector fields}
As presented here, impulsive semiflows are described by certain parameters, namely an initial vector field $X\in \mathfrak{X^r}(M)$, an impulsive region $D$ and an impulsive map $I: I \to M$.  Example~\ref{ex:prey} illustrates that  perturbations of the impulsive maps are not enough to guarantee the full strength of
the general density theorem, namely the denseness of hyperbolic periodic orbits in the whole non-wandering set for a typical impulse. The reason is that 
perturbations of the impulse are localized, and may not affect the space of orbits that does not reach the impulsive region.
A technical issue that appears is that 
the Poincar\'e maps considered in Definition~\ref{defhyp} are of the form ~\eqref{defPo00}, that is, there exist $C^1$-functions $\tau_i: V_i \to\mathbb R^+$ so that $P_i(x) =I \circ \varphi_{\tau_i(x)}(x)$ for every $1\le i \le N$ and $x\in V_i$. By differentiable dependence of the flow on the initial conditions and parameters, it is not hard to conclude that each Poincar\'e map $P_i$ is $C^1$-smooth as a function of the vector field. This fact will be useful to develop a perturbative theory of impulsive semiflows, when the impulse is fixed and one perturbs the vector field defining the initial flow. 
In case one fixes the impulse and perturbs the underlying vector field we expect for a general density theorem in full generality to hold.

\section*{Appendix A: Existence of perturbation boxes for impulsive semiflows}

%This Appendix is devoted to a sketch of the proof of Proposition~\ref{thm:cubes}, which follows very closely the proof of  in  \cite[Th\'eor\`eme~A.2]{BC04}, % (see Appendix A in \cite{BC04}). .
In this Appendix we show how Proposition~\ref{thm:cubes} can be deduced from the proof of \cite[Th\'eor\`eme~A.2]{BC04}
(which in it turn generalizes \cite[Th\'eor\`eme 22]{Arn}).
Most lemmas involved in the proof are quite general and are just dependent on the $C^1$-topology and the fact that the space under consideration is a compact manifold, hence admit a straightforward counterpart for self maps acting on the impulsive region $D$. Other lemmas, of dynamical nature, involve the Poincar\'e map 
given by Proposition~\ref{prop:smoothnessP}) and perturbations of  the impulsive map, and the necessary modifications are stressed below. 
% However this constraint can be overpassed since all the perturbations are made via identity perturbations which allows to consider perturbations in the impulsive map which generates perturbations in the flow using the Poincar\'e map that is differentiable (Proposition~\ref{prop:smoothnessP}).
%	\marginpar{\tiny \color{magenta} {\sc if possible} give some naive intuition}
Indeed, as emphasized in \cite[Subsection A.1]{BC04}, the existence of uniform perturbation boxes in Proposition~\ref{thm:cubes} follows from the work of Arnaud \cite{Arn} provided one can obtain two uniform perturbation lemmas which correspond to Lemmas~\ref{l:pertubation} and~\ref{leme:matrix} below.
Indeed, a first ingredient is the following application of a simple $C^1$-perturbation lemma.

\begin{lemma}\label{l:pertubation}
Let $I \in  \mathscr I^{\mathcal T}_{D,\hat D}$ and $\mathcal U$ be a $C^1$-open neighborhood of $I$. There exist constants $\lambda>1$ and $\delta>0$ such that for all pair of points $p,q \in D$ satisfying $d(p,q)< \delta$ there exists a $C^1$-perturbation $h$ of the identity with support in a ball centered at $p$ and radius $\lambda d(p,q)$ such that $h\circ I\in \mathcal{U}$ and $h(p)=q$. 
\end{lemma}

\begin{proof}
Let 
$\mathcal U$ be a $C^1$-open neighborhood of $I$ 
and let $\mathcal V$ be a $C^1$-open neighborhood of the identity map $id: \hat D \to \hat D$ so that 
 $h\circ I\in \mathcal U$ for every $h\in \mathcal V$. Then, 
this lemma is a consequence of \cite[Lemme~A.4]{BC04}, in the context of diffeomorphisms. 
%on a \color{red} compact \color{black} manifold.
Indeed, \cite[Lemme~A.4]{BC04} ensures that if $\mathcal{V}$ is a $C^1$-open neighborhood of the identity map $id: \hat D\to \hat D$ then there exist constants $\lambda>1$ and $\delta>0$ such that for each pair of points $p,q \in \hat D$ so that $d(p,q)< \delta$ there exists a $C^1$-perturbation $h\in \mathcal V$ of the identity map with support in a ball centered at $p$ and of radius $\lambda d(p,q)$  such that $h(p)=q$. This proves the lemma.
\end{proof}

In what follows given a subset  $A$ of a metric space $B$ we denote by $A^c$ the complement of $A$, that is, 
$A^c=B\setminus A$. Recall that we say 
$C\subset D$ is a cube if it is diffeomorphic to a cube in $\mathbb R^{d-1}$ by a chart.
The second main ingredient is as follows:

\begin{lemma}\label{leme:matrix}
%\marginpar{\tiny \color{red} do we need to avoid points mapped in the boundary of $\hat D$ (or $D$) for N iterates? }
Let $I \in  \mathscr I^{\mathcal T}_{D,\hat D}$ and $\mathcal U$ be a $C^1$-open neighborhood of $I$. Given $\vep,\eta\in (0,1)$ 
there exists an integer $N\ge 1$ such that: for every $p\in D$, every chart containing $p$ in its domain 
and every cube $C\subset D$ 
%(i.e. diffeomorphic to a cube in $\mathbb R^{d-1}$ by a chart) 
such that 
$\{I^{-1}\circ P_I^\ell \circ I (C)\colon 0\le \ell \le N\}$ is a pairwise disjoint collection of subsets in $D$,
and every pair of points $p,q\in C$ there exists a sequence of points $(x_k)_{k=0}^N$
in $(1+\vep) C$ satisfying $x_0=p$, $x_N=q$ and 
$$
d(I^{-1}\circ P_I^k\circ I(x_k), I^{-1}\circ P_I^k\circ I(x_{k-1}) ) \leq \eta \cdot \text{dist}(f^k((1+\vep)C), f^k((1+2\vep) C)^c)
$$ 
for any $1\le k \le N$. 
\end{lemma}

\begin{proof}
This lemma is a consequence of Proposition A.3 in \cite{BC04} applied to the $C^1$-map on $D$ given by $f_I=I^{-1}\circ P_I\circ I$. 
%\marginpar{\tiny \color{red} note that $I^{-1}\circ P_I\circ I$ is not a diffeo, but A.3 requires in the assumptions} 
Indeed, the latter relies on \cite[Proposition A.2]{BC04} which consider linear maps and guarantees that for each $K>0$ and $\vep,\eta\in (0,1)$ fixed there exists  $N\ge 1$ such that for any collection $(T_k)_{k=0\dots N}$
of matrices in $GL(d-1,\mathbb R)$ satisfying 
\begin{equation}\label{eq:normsTk}
\|T_k \circ T_{k-1}^{-1}\|\le K\quad \text{and} \quad \|T_{k-1}\circ T_k^{-1}\|\le K
\end{equation}
for any $0\le k \le N$ so that the following holds: there exists a cube $C\subset D$  such that for any pair of points $p,q \in C$ there exists a sequence of points $(p_k)_{k=0}^N$
in $(1+\vep) C$ satisfying $x_0=p$, $x_N=q$ and 
$$
d(T_k(x_k), T_k(x_{k-1})) \leq \eta \cdot \text{dist}(T_k((1+\vep)C), T_k((1+2\vep) C)^c).
$$ 
Similarly, the lemma follows as an application of Proposition A.2 in \cite{BC04} by considering the $C^1$-map 
$I^{-1}\circ P_I\circ I$ and the linear maps $T_k=DI^ {-1} \circ DP_I^k\circ DI$, for $1\le k \le N$.
provided one guarantees these matrices are uniformly bounded in the sense of ~\eqref{eq:normsTk}.
It is enough to observe that, recalling ~\eqref{defPo00},
\begin{align*}
T_k\circ T_{k-1}^{-1}(x) 
	& = DI^ {-1} \circ DP_I^k \circ DI \circ (DI^ {-1} \circ DP_I^{k-1} \circ DI)^{-1}(x) \\
	& = DI^ {-1}(I\circ f_I^k(x)) \circ DP_I (I \circ f_I^{k-1}(x)) \circ DI(f^{k-1}_I(x)) \\
	& = DI^ {-1}(I\circ f_I^k(x)) \circ D(I\circ \varphi_{\tau(I\circ f_I^{k-1}(x))}) (I \circ f_I^{k-1}(x)) \circ DI(f^{k-1}_I(x)) \\
	% & =  \frac{\partial}{\partial x} [\varphi_{\tau(I\circ f_I^{k-1}(x) )}] (I \circ f_I^{k-1}(x)) \circ DI(f^{k-1}_I(x)) \\	
	& =  \frac{\partial \varphi_{\tau(\cdot )}}{\partial x}  ( I\circ f_I^{k-1}(x)) \circ DI(f^{k-1}_I(x)) 
\end{align*}
for every $x\in D$ such that $P_I^k(x)$ is well defined for $1\le k \le N$. Recalling the notation $\varphi(t,x)=\varphi_t(x)$,
observe that,
\begin{align*}
\frac{\partial}{\partial x} \varphi(\tau(x),x) 
= \frac{\partial\varphi}{\partial t} (\tau(x),x)  
\cdot  \frac{d\tau}{d x} (x) 
	+ \frac{\partial\varphi}{\partial x} (\tau(x),x).
	\end{align*}
Since the norm of the vector field $X =\frac{\partial\varphi}{\partial t}  $ and the norm of $\frac{\partial\varphi}{\partial x}$ are continuous functions on a  compact 
metric space, hence  bounded, we conclude that there exists $A>0$ (depending on the original vector field $X$) so that
$$
\|T_k\circ T_{k-1}^{-1}(x) \| \le A \cdot \max\{\|I\|_{C^1}, \|I^{-1}\|_{C^1}\} \cdot  |\frac{\partial \tau}{\partial x} ( I\circ f_I^{k-1}(x))|, 
	 $$
which is uniformly bounded by the assumption
that  $I\in \mathscr I^{\mathcal T}_{D}$.
%where $A$  is a uniform bound for the derivative of the Poincaré map $P_I$.  
%The uniform bound is guaranteed by \cite{Poincare}, where it is proven that $\tau_1$ is a $C^1$-map. 
%
As the estimates for the matrices $T_{k-1}\circ T_k^{-1}$ are identical, leading to the same bound above, 
this completes the proof of the lemma.
\end{proof}

\color{black}

%\begin{remark}
%As impulsive semiflows combine information from the continuous action of the flow and a discrete time impulsive map, it follows from the proof that the uniform time $N\ge 1$ in  Lemma~\ref{leme:matrix}
%depends both on the neighborhood $\mathcal U$ of the impulsive map $I$
%and also on the original flow $\varphi$.
%\end{remark}

\medskip

\color{black}

\bigskip

%%%%%%%%%%%%%%%%%%%%%%%%%%%%%%%%%%%%%%%%%%%%%%%%%%%%%

\medskip{\bf Acknowledgements.} 
The authors are grateful to S. Crovisier for useful comments about $C^1$-closing lemmas for maps. 
JS was partially supported by CMUP (UID/MAT/00144/2013), 
the grant E-26/010/002610/2019, Rio de Janeiro Research Foundation (FAPERJ), and by the Coordena\c c\~ao de Aperfeiçoamento de Pessoal de N\'ivel Superior - Brasil (CAPES), Finance Code 001.
MJT was partially financed by Portuguese Funds through FCT (Funda\c c\~ao para a Ci\^encia e a Tecnologia) through the research 
Projects UIDB/00013/2020 
(DOI: 10.54499/UIDB/00013/2020) 
and UIDP/00013/2020 
(DOI: 10.54499/UIDP/00013/2020).
PV was partially supported by 
by Funda\c c\~ao para a Ci\^encia e Tecnologia (FCT) - Portugal, through the grant CEECIND/03721/2017 of the Stimulus of Scientific Employment, Individual Support 2017 Call and
CMUP (UID/MAT/00144/2013).
%The three authors were also partially supported by the Project
%``New trends in Lyapunov exponents"(PTDC/MAT-PUR/29126/2017).
%%%%%%%%%%%%%%%%%%%%%%%%%%%%%%%%%%%%%%%%%%%%%%%%%%%%%%%%%%%%%%%%%%%%%%%%

\end{document}